\documentclass[final]{siamart1116}

\usepackage{url}

\usepackage{graphicx}
\usepackage{float}
\usepackage{color}
\usepackage{amsmath}
\usepackage{amssymb}
\usepackage{amsfonts}
\usepackage{multirow}
\usepackage{subfigure}
\usepackage{bm}
\usepackage[version=4]{mhchem}
\usepackage{mathtools}
\usepackage{makecell}
\usepackage{xr}
\renewcommand{\figurename}{\textbf{Fig.}}
\numberwithin{equation}{section}  
\newtheorem{Example}{Example}[section]

\DeclareMathOperator{\rank}{rank}
\DeclareMathOperator{\diag}{diag}

\DeclareMathOperator{\argmin}{argmin}

\DeclareMathOperator{\arctanh}{arctanh}
\DeclareMathOperator{\HH}{\sf H}
\DeclareMathOperator{\T}{\sf T}

\def\ee{\mathrm{e}}
\def\ii{\mathrm{i}}
\makeatletter
\tagsleft@false\let\veqno\@@eqno
\makeatother
\def\sss{\scriptscriptstyle}
\def\hm{\hphantom{-}}
\title{Doubling Algorithm for The Discretized Bethe-Salpeter Eigenvalue Problem}
\author{Zhen-Chen Guo\thanks{Department of Applied Mathematics, National Chiao Tung University, Hsinchu 300, Taiwan; 
	\texttt{guozhch06@gmail.com}} \and
Eric King-wah Chu\thanks{School of Mathematics, Monash University, 9 Rainforest Walk, Victoria 3800, Australia; {\tt eric.chu@monash.edu}} \and
Wen-Wei Lin\thanks{Department of Applied Mathematics, National Chiao Tung University, Hsinchu 300, Taiwan; {\tt wwlin@math.nctu.edu.tw}}
}  

\begin{document}
\date{}

\maketitle

\begin{abstract}
The discretized Bethe-Salpeter eigenvalue problem arises in the Green's function evaluation in many body physics and quantum chemistry. Discretization leads to a matrix eigenvalue problem for $H \in \mathbb{C}^{2n \times 2n}$ with a Hamiltonian-like structure. After an appropriate transformation of $H$ to a standard symplectic form, the structure-preserving doubling algorithm, originally for algebraic Riccati equations, is extended for the discretized Bethe-Salpeter eigenvalue problem. Potential breakdowns of the algorithm, due to the ill condition or singularity of certain matrices, can be avoided with a double-Cayley transform or a three-recursion remedy. A detailed convergence analysis is conducted for the proposed algorithm, especially on the benign effects of the double-Cayley transform. Numerical results are presented to demonstrate the efficiency and structure-preserving nature of the algorithm. 
\end{abstract}



\section{Introduction}

The  Bethe-Salpeter equation (BSE)~\cite{sb51} arises in the Green's function evaluation in many body physics, which
is the state-of-art model to  describe  electronic excitation and  molecule absorption~\cite{c95,kbdbab14,Leng:16,orr02,phssm13,prg13,rhl12,roro02,rts13a,rts13b,rl00,rg84,ssf98,w90}.
In the quantum chemistry and material science communities, the optical absorption spectrum of the BSE is an important and powerful tool for the characterization of different materials. In particular, the comparison of the computed 
and measured spectra helps to interpret experimental data and validate corresponding theories and models. 
It is generally known that good agreement between the theory and the experimental data can only be achieved 
by taking into account the interacting electron-hole pairs or {\em excitons}.
This is the case for the BSE which is derived from the coupling of the electrons and their corresponding holes.

After discretization, the BSE becomes the Bethe-Salpeter  eigenvalue  problem (BS-EVP):
\begin{equation}  \label{bsevp}
Hx\equiv
\begin{bmatrix}     
 \ \ A & \ \ B \\ \\
-\overline{B}  & -\overline{A}
\end{bmatrix}x	
=\lambda x,
\end{equation}
for $x \neq 0$, where $A, B \in\mathbb{C}^{n\times n}$ satisfy
$A^{\HH}=A,\ B^{\T}=B$.
Here $(\cdot)^{\HH}$ and $(\cdot)^{\T}$ denote the conjugate transpose and the transpose of matrices, respectively.
It can be shown~\cite{bfy15} that any eigenvalue $\lambda$ comes in quadruplets $\{\pm \lambda, \pm \overline{\lambda}\}$ (except for the degenerate cases when $\lambda$ is purely real or imaginary, or zero). Further details on the BS-EVP can be found in~\cite{peter:16,peter:15,yang:16} and the references therein.

In principle, all possible excitation energies and absorption spectra are sought although some excitations are more probable than others. The associated likelihood is measured by the spectral density or the density of states of $H$, defined as the number of eigenvalues per unit energy interval:
\[
\phi(\omega) = \frac{1}{2n} \sum_{j=1}^{2n} \delta(\omega - \lambda_j),
\]
where $\delta$ is the Dirac-delta function and $\lambda_j \in \lambda(H)$, the spectrum of $H$. Also of interest is the optical absorption spectrum:
\[
	\epsilon^+(\omega) = \sum_{j=1}^n \frac{(d_r^{\HH} x_j)(y_j^{\HH} d_l)}{y_j^{\HH} x_j} \delta (\omega - \lambda_j),
\]
where $x_j$ and $y_j$ are, respectively, the right- and left-eigenvectors corresponding to $\lambda_j>0$, and $d_r$ and $d_l$ are the dipole vectors. Evidently, to estimate these quantities, we require {\it all} the eigenvalues $\lambda_j$ and the associated eigenvectors $x_j$ and $y_j$. To complicate computations further, $A$ and $B$ are often high in dimensions (for systems with many occupied and unoccupied states) and generally dense.

In spite of the significance of the BS-EVP \eqref{bsevp}, only a few publications exist on its numerical solution, all under {\it additional} assumptions.
Some remarkable discoveries have been made in~\cite{peter:16,peter:15,yang:16}
under the condition that $\Gamma H$ is positive
definite with $\Gamma =\diag(I_{n},\,  -I_{n})$. Few general and efficient  methods have been
proposed to solve the BS-EVP \eqref{bsevp}.
All  methods proposed in~\cite{peter:16,peter:15,yang:16} are designed  
for the linear response eigenvalue problem,  under the extra assumptions that
$A, B\in\mathbb{R}^{n\times n}$ and $A \pm B$ are symmetric positive definite. 
Low-rank or tensor approximations~\cite{peter:16,peter:15} have been applied to handle the 
high computational demand but these techniques require additional structures on $H$.
Based on the equivalence of  the BS-EVP and a real Hamiltonian eigenvalue problem, 
Shao et al.~\cite{yang:16} put forward an efficient parallel approach
to compute the eigenpairs  corresponding to all the positive eigenvalues. 
Remarkable contributions have also been made for the numerical solution of the 
related linear response eigenvalue problem~\cite{bl12,bl13}.

\subsection*{Contributions} We solve the  {\it general} BS-EVP \eqref{bsevp}, 
without assuming  $\Gamma H$ being positive definite. We propose a doubling algorithm (DA) for the 
BS-EVP in two recursions. To deal with potential breakdowns, 
we design the double-Cayley transform (DCT) and a three-recursion remedy. The DCT 
reverses at worst two steps of the DA if there exist some complex eigenvalues and not at all if all eigenvalues are real.  
In the rare occasions that the DCT fails, the more expensive three-recursion remedy can be applied, without changing the convergence radius. 
Our DA preserves the special structure of the eigen-pairs.

\subsection*{Organization} 
Some preliminaries are presented in Section~2 and
our method is developed in Section~3.
We present some illustrative numerical results in Section~4
before the conclusions in Section~5. The Appendix contains two technical lemmas. 

\section{Preliminaries}

We denote the column space, the null space, the spectrum and the set of singular values by 
$\mathcal{R}(\cdot)$,  $\mathcal{N}(\cdot)$, $\lambda(\cdot)$ and $\sigma (\cdot)$ respectively. 
By $M\oplus N$ or $\diag(M,N)$, we denote $\begin{bmatrix}
M & \mathbf 0\\ \mathbf 0&N
\end{bmatrix}$. Similarly, we define $\bigoplus_{j} M_j$. 
The MATLAB expression $M(k:l, s:t)$ denotes the submatrix of $M$ containing elements in 
rows $k$ to $l$ and columns $s$ to $t$. Also, the $i$th column of the identity matrix $I$ is $e_i$ and
\[
J \equiv \begin{bmatrix}&I_n\\ \\-I_n&\end{bmatrix},
\qquad \Gamma \equiv \begin{bmatrix}I_n&\\ \\&-I_n\end{bmatrix},
\qquad \Pi \equiv \begin{bmatrix} &I_n\\ \\I_n& \end{bmatrix}.
\]

\begin{definition} 
The matrix pair  $(M, \, L)$ with $M, L\in\mathbb{C}^{2n\times 2n}$ is a symplectic pair if and only if $MJM^{\T}=LJL^{\T}$.
\end{definition}

\begin{definition}
The matrix pair $(M,\,  L)$ is in the first standard symplectic form (SSF-1) if and only if   
\[
M=\begin{bmatrix}E&\mathbf 0\\ \\ F&I_n\end{bmatrix}, \qquad
L=\begin{bmatrix}I_n&K\\ \\ \mathbf 0&E^{\T}\end{bmatrix},  \]
	with  $E, F\equiv F^{\T}, K\equiv K^{\T}\in\mathbb{C}^{n\times n}$. 
\end{definition}

\begin{definition}
Let $M, L\in\mathbb{C}^{2n\times 2n}$ and denote $\mathcal{N}(M, L)\equiv$
\[
\left\{[M_{*},\, L_{*}]: M_{*}, L_{*}\in\mathbb{C}^{2n\times 2n}, \
\rank([M_{*},\, L_{*}])=2n, \ [M_{*},\, L_{*}][L^{\T},\, -M^{\T}]^{\T}=0 \right\},
\]
which is nonempty. The action 
$(M, \, L) \longrightarrow  (\widetilde{M}, \, \widetilde{L}) = (M_{*}M, \, L_{*}L)$ 
is called a {\it doubling transformation} of $(M, L)$ for some  $[M_{*},\, L_{*}] \in\mathcal{N}(M, L)$.
\end{definition}

Next we consider the properties of the doubling transformation. 
\begin{lemma}(\cite[Theorem~2.1]{lx06})\label{lemme-doubling}
Let $(\widetilde{M},\,  \widetilde{L})$ be the result of a doubling transformation of $(M,\, L)$, where
$M, L$, $\widetilde{M}, \widetilde{L}\in\mathbb{C}^{2n\times 2n}$,
we have
\begin{enumerate}
  \item [{\em (1)}] $(\widetilde{M}, \, \widetilde{L})$ is a symplectic pair provided that  $(M,\, L)$ is one; and
  \item [{\em (2)}] if $MU=LUR$ and $MVS=LV$ for some $U, V\in\mathbb{C}^{2n\times  l }$ and $R, S\in\mathbb{C}^{ l \times  l }$,
                    then $\widetilde{M}U=\widetilde{L}UR^2$ and $\widetilde{M}VS^2=\widetilde{L}V$.
\end{enumerate}
\end{lemma}
In other words, doubling transformations preserve symplecticity and deflating subspaces as well as square eigenvalues of matrix pairs. 

\begin{lemma}\label{lemma-H-Hermitian}
It holds that  $H\Pi =-\Pi {\overline{H}}$ and $\Gamma H \Gamma=H^{\HH}$.
\end{lemma}
\begin{proof}
It can be verified directly.
\end{proof}

\begin{lemma}\label{lemma-eigenpairs}
Assume that $HZ=ZS$ with $Z\in\mathbb{C}^{2n\times  l }$ and  $S\in\mathbb{C}^{ l \times  l }$,
then we have
$H (\Pi\overline{Z})=(\Pi\overline{Z})(-\overline{S})$ and $(Z^{\HH}\Gamma) H=S^{\HH}(Z^{\HH}\Gamma)$.
\end{lemma}
\begin{proof}
The results directly  follow from  Lemma~\ref{lemma-H-Hermitian}.
\end{proof}

If $S$ in Lemma~\ref{lemma-eigenpairs} possesses the spectrum 
$\lambda(S)=\{\lambda, \ldots, \lambda\}$ (repeated $l$ times), Lemmas~\ref{lemma-H-Hermitian} and~\ref{lemma-eigenpairs}
imply that $-\lambda$, $\overline{\lambda}$ and $-\overline{\lambda}$  are also the
eigenvalues of $H$ with the same algebraic and geometric multiplicities.
Provided that $HX_j=X_jS_j$ with $X_j\in\mathbb{C}^{2n\times l_{j}}$ and  $S_j\in\mathbb{C}^{l_j\times l_j}$
for $j=1, 2$, Lemma~\ref{lemma-eigenpairs} further implies that
$(X_2^{\HH}\Gamma X_1) S_1 = X_2^{\HH}\Gamma H X_1=S_2^{\HH}(X_2^{\HH}\Gamma X_1)$ and
$(X_2^{\T}\Pi \Gamma) X_1S_1 = (X_2^{\T}\Pi \Gamma)HX_1=(-S_2^{\T})(X_2^{\T}\Pi \Gamma X_1)$,
or equivalently 
\[
(X_2^{\HH}\Gamma X_1) S_1-S_2^{\HH}(X_2^{\HH}\Gamma X_1)=0= 
(X_2^{\T}\Pi \Gamma X_1) S_1 + S_2^{\T}(X_2^{\T}\Pi \Gamma X_1). 
\]
Apparently, when $\lambda(S_1)\cap\lambda(\overline S_2)=\emptyset$, we have  $X_2^{\HH}\Gamma X_1=0$;
when $\lambda(S_1)\cap\lambda(-S_2)=\emptyset$, we have $X_2^{\T}\Pi\ \Gamma X_1=0$. 
By Lemmas~\ref{lemma-H-Hermitian} and~\ref{lemma-eigenpairs}, we can then deduce the 
eigen-decomposition result of $H$ for the convergence proof.

Temporarily assume that there is no purely imaginary  nor zero eigenvalues  for $H$,  
$\lambda_j\neq\lambda_k$ for $j\neq k$ and 
\begin{align*}
\lambda(H)=&\{\underbrace{\lambda_1, \ldots, \lambda_1}_{l_1}, \underbrace{\overline{\lambda}_1, \ldots, \overline{\lambda}_1}_{l_1},
\underbrace{-\overline{\lambda}_1, \ldots, -\overline{\lambda}_1}_{l_1}, \underbrace{-\lambda_1, \ldots, -\lambda_1}_{l_1}, \ldots, \\
&\underbrace{\lambda_s, \ldots, \lambda_s}_{l_s}, \underbrace{\overline{\lambda}_s, \ldots, \overline{\lambda}_s}_{l_s},
\underbrace{-\overline{\lambda}_s, \ldots, -\overline{\lambda}_s}_{l_s}, \underbrace{-\lambda_s, \ldots, -\lambda_s}_{l_s}, \\
&\underbrace{\lambda_{s+1}, \ldots, \lambda_{s+1}}_{l_{s+1}}, \underbrace{-\lambda_{s+1}, \ldots, -\lambda_{s+1}}_{l_{s+1}}, \ldots,
\underbrace{\lambda_{t}, \ldots, \lambda_{t}}_{l_{t}}, \underbrace{-\lambda_{t}, \ldots, -\lambda_{t}}_{l_{t}}\}, 
\end{align*}
where  $\lambda_j\in\mathbb{C}$  with (i)  $\Re(\lambda_{j})\Im(\lambda_j)\neq0$ 
and $\Re(\lambda_j)<0$ for $j=1, \ldots, s$,
and (ii) $\Im(\lambda_j)=0$ and $\lambda_{j}<0$ for  $j=s+1, \ldots, t$. Subsequently, we have the following result.

\begin{lemma}\label{lemma-decomp}
Suppose  that  no purely imaginary  nor zero eigenvalues exist for $H$.   
Then there exist 
\begin{align*}
	X&= [X_1, Y_1, \cdots, X_s, Y_s;\, X_{s+1}, \cdots, X_{t}] \in\mathbb{C}^{2n\times n},
	\\
	S&=\diag(S_1, R_1, \ldots, S_s, R_s;\, S_{s+1}, \ldots, S_{t})\in\mathbb{C}^{n\times n}
\end{align*} 
with $X_j\in\mathbb{C}^{2n\times l_j}$, $S_{j}\in\mathbb{C}^{l_j\times l_j}$,
$\lambda(S_{j})=\{\lambda_j, \ldots, \lambda_j\}$ ($j=1, \ldots, t$),
$Y_{j}\in\mathbb{C}^{2n\times l_j}$, $R_{j}\in\mathbb{C}^{l_j\times l_j}$ and 
$\lambda(R_{j})=\{\overline{\lambda}_j, \ldots, \overline{\lambda}_j\}$ ($j=1, \ldots, s$), such that
\[
H [X,\, \Pi\overline{X} ] = [X,\, \Pi\overline{X}] \diag(S,\, -\overline{S}), \ \ \ 
[X,\, \Pi\overline{X}]^{\HH}\Gamma [X,\, \Pi\overline{X}]=\diag(D,\, -\overline{D}), \\ 
\]
where $D=\diag(D_1, \ldots, D_{s};\, D_{s+1}, \ldots, D_{t})$,   
\begin{align*}
& D_j=\begin{bmatrix}\mathbf 0&X_j^{\HH}\Gamma Y_j\\ Y_j^{\HH}\Gamma X_j&\mathbf 0 \end{bmatrix}\in\mathbb{C}^{2l_j\times 2l_j} \ \ \ (j=1, \ldots, s), \\
& D_j=X_j^{\HH}\Gamma X_j\in\mathbb{C}^{l_j\times l_j} \ \ \ (j=s+1, \ldots, t).
\end{align*}
\end{lemma}

Obviously, $D$ (in Lemma~\ref{lemma-decomp}) is a nonsingular Hermitian matrix.  
Consequently, we can choose $X$ which satisfies  
$[X,\, \Pi\overline{X}]^{\HH}\Gamma [X,\, \Pi\overline{X}] =\Gamma$. This leads to  $X^{\HH} \Gamma X = I_n$ and 
$X^{\T}\Gamma \Pi X=0$, implying that $X(1:n,1:n)\in\mathbb{C}^{n\times n}$ 
is nonsingular with singular 
values no less than unity and $X(1:n,1:n)^{\T} X(n+1:2n,1:n)$ is complex  symmetric. 
 
Next consider the case when there exist some purely imaginary eigenvalues for $H$. 
We further assume that the partial multiplicities (the sizes of the Jordan blocks) of $H$ associated 
with the purely imaginary eigenvalues are all even. Let $\ii\omega_1, \cdots, \ii\omega_q$ be the different 
purely imaginary eigenvalues with Jordan blocks   
$J_{2p_{r,j}}(\ii\omega_j)\in \mathbb{C}^{2p_{r,j} \times 2p_{r,j}}$ 
for $r=1, \cdots, l_j$ and $j=1, \cdots, q$.  Then  there exist $W_{r,j}, Z_{r,j}\in 
\mathbb{C}^{2n\times p_{r,j}}$   
such that  
\begin{align*}
	&	
	H \left[ W_{1,1}, Z_{1,1}; \cdots; W_{l_1,1}, Z_{l_1,1} \,\vrule \,  
	\cdots \, \vrule \, W_{1,q}, Z_{1,q}; \cdots; W_{l_q,q}, Z_{l_q,q} \right]
	\\
	=&
	\begin{multlined}[t]
		\left[ W_{1,1}, Z_{1,1}; \cdots; W_{l_1,1}, Z_{l_1,1} \,\vrule \,  
		\cdots \, \vrule \, W_{1,q}, Z_{1,q}; \cdots; W_{l_q,q}, Z_{l_q,q} \right]
		\cdot\left[ \bigoplus_{j=1}^q\bigoplus_{r=1}^{l_j}  J_{2p_{r,j}}(\ii\omega_j)\right].
	\end{multlined}
\end{align*}
With $X\in \mathbb{C}^{2n\times n_1}$ and $S\in \mathbb{C}^{n_1\times n_1}$ and by Lemma~\ref{lemma-decomp}, 
we obtain 
\begin{align}\label{eq:eigen_decomp}
	H\left[ X, W_{\omega},  \Pi \overline{X}, Z_{\omega} \right]
	=
	\left[ X, W_{\omega},  \Pi \overline{X}, Z_{\omega} \right]
	\widetilde{S},
\end{align}
where $n_1+\sum_{j=1}^{q}\sum_{r=1}^{l_j}p_{r,j}=n$,  and
\begin{align*}
	W_{\omega}
	&=
	\left[ 
		W_{1,1}, \cdots, W_{l_1,1};  \cdots; W_{1,q}, \cdots, W_{l_q,q}
	\right],
	\\
	Z_{\omega}
	&=
	\left[
		Z_{1,1}, \cdots, Z_{l_1,1}; \cdots; Z_{1,q}, \cdots, Z_{l_q,q}
	\right],
	\\
	J_{\omega}&= \bigoplus_{j=1}^q\bigoplus_{r=1}^{l_j}   J_{p_{r,j}}(\ii\omega_j), 
	\qquad \qquad \Omega_{\omega}=  \bigoplus_{j=1}^q\bigoplus_{r=1}^{l_j}e_{p_{r,j}}e_1^{\T}, 
	\\
	J_{2p_{r,j}}(\ii\omega_j) & \equiv \begin{bmatrix}
		J_{p_{r,j}}(\ii\omega_j) & e_{p_{r,j}}e_1^{\T}
		\\
		0& J_{p_{r,j}}(\ii\omega_j)
	\end{bmatrix}, \ \ \ \widetilde{S} \equiv \begin{bmatrix}
		S & & &\\
		& J_{\omega} && \Omega_{\omega}\\
		\\
		& & -\overline{S}&\\
		& & & J_{\omega}
	\end{bmatrix}.  
\end{align*}

\section{Doubling Algorithm}\label{doublesec}

We now generalize the structure-preserving doubling algorithm (SDA) in~\cite{cfl,cflw,lcjl1,lcjl2} 
to  the DA for the BS-EVP. 

\subsection{Initial Symplectic Pencil} We transform $H$ to a symplectic pair  $(M,\, L)$ in the SSF-1 {\it \`a la} Cayley.

\begin{lemma}
For $\alpha\in\mathbb{R}$, the matrix pair
$(H+\alpha I_{2n}, \, H-\alpha I_{2n})$ is  symplectic.
\end{lemma}
\begin{proof}
The result can be deduced from $(HJ)^{\T}=HJ$.
\end{proof}

\begin{theorem} \label{theorem-SSF-1}
Select  $\alpha\in\mathbb{R}$ such that both 
$\alpha I_n-A$ and 
$R \equiv I_n-(\alpha I_n-\overline{A})^{-1}\overline{B}(\alpha I_n-A)^{-1}B$ 
are nonsingular.
There exists a nonsingular matrix $G\in\mathbb{C}^{2n\times 2n}$
such that $[G(H+\alpha I_n),\,  G(H-\alpha I_n)]$
is  a symplectic pair in SSF-1, with 
\begin{equation} \label{MLalpha}
M_{\alpha}\triangleq G(H+\alpha I_n)=\begin{bmatrix}E_{\alpha}&\mathbf 0\\ \\F_{\alpha}&I_n\end{bmatrix}, \, \, \,  
L_{\alpha}\triangleq G(H-\alpha I_n)=\begin{bmatrix}I_n&\overline{F}_{\alpha}\\ \\ \mathbf 0&\overline{E}_{\alpha}\end{bmatrix},
\end{equation}
where $E_{\alpha}, F_{\alpha}\in\mathbb{C}^{n \times n}$ satisfy $E_{\alpha}^{\HH}=E_{\alpha}$ and $F_{\alpha}^{\T}=F_{\alpha}$.
\end{theorem}

\begin{proof}
Let $H_{\sss\pm} \equiv H \pm \alpha I_{2n}$, $A_{\sss\pm} \equiv A \pm \alpha I_n$, 
\begin{align*}
	G_1=\begin{bmatrix} A_{\sss -}^{-1} & \mathbf 0\\
&\\
\overline{B} A_{\sss -}^{-1} &  I_n\end{bmatrix},
& \ \ \ 
G_2=\begin{bmatrix}I_n& A_{\sss -}^{-1}BR^{-1} \overline{A}_{\sss -}^{-1}\\
&\\
 \mathbf 0 & -R^{-1} \overline{A}_{\sss -}^{-1}\end{bmatrix},
\end{align*}
and $G=G_2G_1$. We obtain  
\begin{align*}
& G_1H_{\sss +} =\begin{bmatrix} A_{\sss -}^{-1}A_{\sss +} & A_{\sss -}^{-1}B\\ \\
2\alpha\overline{B} A_{\sss -}^{-1} & -\overline{A}_{\sss -} R \end{bmatrix}, \ \ \ 
G_2G_1H_{\sss +} =\begin{bmatrix}E_{\alpha} & \mathbf 0\\ \\ F_{\alpha}& I_n\end{bmatrix}, \\ 
& G_1H_{\sss -} =\begin{bmatrix}I_n & A_{\sss -}^{-1}B\\ \\\mathbf 0& -\overline{A}_{\sss -} R-2\alpha I_n\end{bmatrix},  \ \
G_2G_1H_{\sss -} =\begin{bmatrix}I_n & \overline{F}_{\alpha} \\ \\ \mathbf 0& \overline{E}_{\alpha}\end{bmatrix},
\end{align*}
with   
\begin{equation}\label{EFalpha}
E_{\alpha}=I_n + 2\alpha \overline{R}^{-1} A_{\sss -}^{-1}, \ \ \  
F_{\alpha}=-2\alpha \overline{A}_{\sss -}^{-1}\overline{B}\, \overline{R}^{-1} A_{\sss -}^{-1}. 
\end{equation}
Furthermore, since $A^{\HH}=A$ and $B^{\T}=B$,  we have 
\begin{align*}
E_{\alpha}^{\HH} &=I_{n} + 2\alpha A_{\sss -}^{-1}R^{-\T}=
I_n+2\alpha ( A_{\sss -}^{-1}-B \overline{A}_{\sss -}^{-1}\overline{B} )^{-1}=E_{\alpha}, \\
F_{\alpha}^{\T} &=-2\alpha \overline{A}_{\sss -}^{-1}
( I_n-\overline{B} A_{\sss -}^{-1}B \overline{A}_{\sss -}^{-1} )^{-1}\overline{B} A_{\sss -}^{-1}=F_{\alpha}, 
\end{align*}
i.e., $E_{\alpha}$  and $F_{\alpha}$ are Hermitian and complex symmetric, respectively.
Lastly, we have 
\[
(G H_{\sss\pm}) J (G H_{\sss\pm})^{\T}=\begin{bmatrix}&E_{\alpha}\\ \\-\overline{E}_{\alpha}\end{bmatrix},
\]
implying that $[G(H+\alpha I_n), \, G(H-\alpha I_n)]$ is a symplectic pair in SSF-1.
\end{proof}

The following   lemma summarizes the eigen-structure of   $(M_\alpha, \, L_\alpha)$ 
in relation to that of  $H$, neglecting the simple proof.

\begin{lemma}\label{eigenML0}
Let 
\begin{align}\label{eigenH}
H [X_1^{\T},\, X_2^{\T}]^{\T}=[X_1^{\T},\, X_2^{\T}]^{\T}S
\end{align}
for some $X_1, X_2\in\mathbb{C}^{n\times l}$,  $S\in\mathbb{C}^{l\times l}$ and 
$\alpha \notin \lambda(H)$, then we have
\[
M_{\alpha}[X_1^{\T},\, X_2^{\T}]^{\T}
=L_{\alpha}[X_1^{\T},\, X_2^{\T}]^{\T}S_{\alpha},
\]
with $S_{\alpha}\equiv (S-\alpha I_l)^{-1}(S+\alpha I_l)$,  
where $S_{\alpha}- \alpha I_l$ is nonsingular. 
\end{lemma}

Intrinsically,  the DA proposed below   
requires both $E_{\alpha}$ and 
$I_n-F_{\alpha}\overline{F}_{\alpha}$ to be nonsingular.  
Lemma~\ref{l34} and Theorems~\ref{lemma-E0} and~\ref{lemma-F0} below indicate that a 
small $\alpha$ could achieve such a goal.  
Moreover, for $\lambda\in\lambda(H)$, we have 
$(\lambda+\alpha)/(\lambda-\alpha) \in \lambda (S_{\alpha})$. For the efficiency of  the DA, we desire a small 
$\left|(\lambda+\alpha)/(\lambda-\alpha)\right|$ for $\Re(\lambda) <0$. Hence  when $|\alpha|>\rho(H)$ 
(the spectral radius of $H$), we desire 
$|\alpha|$  to be minimized.

\begin{lemma} \label{l34} 
Let $\alpha>\|H\|_F$, then $\alpha I_n-A$ is positive definite and 
$R\equiv I_n-(\alpha I_n-\overline{A})^{-1}\overline{B}(\alpha I_n-A)^{-1}B$ is nonsingular, with 
$\|R^{-1}\|_2\leq \left[ 1-\|(\alpha I_n-A)^{-1}\|_2^2\|B\|_2^2\right]^{-1}$.
\end{lemma}

\begin{proof}
When $\|A\|_F<\|H\|_F<\alpha$, $\alpha I_n-A$ is positive definite Hermitian.
Since  $\alpha>\|H\|_F\geq\|A\|_F+\|B\|_F$, we have 
$(\alpha-\omega_1)^{-1} \leq (\alpha-\|A\|_F)^{-1} < \|B\|_F^{-1}$ with $\omega_1$
being the largest eigenvalue of $A$. In addition, with 
$\|(\alpha I_n-A)^{-1}\|_2= (\alpha-\omega_1)^{-1}$, we have    
$\|(\alpha I_n-A)^{-1}B\|_2\leq\|(\alpha I_n-A)^{-1}\|_2\|B\|_2= (\alpha-\omega_1)^{-1} \|B\|_2<1$.   
This implies $\|(\alpha I_n-\overline{A})^{-1}\overline{B}(\alpha I_n-A)^{-1}B\|_2\leq
\|(\alpha I_n-A)^{-1}\|_2^2\|B\|_2^2<1$ and  our results. 
\end{proof}

\begin{theorem}\label{lemma-E0}
As defined in 
\eqref{EFalpha}, $E_{\alpha}$ is nonsingular when $\alpha>\|H\|_F$. 
\end{theorem}

\begin{proof}
Denote the largest and smallest eigenvalues of $A$ by $\omega_1$ and $\omega_n$, respectively.
With $\alpha>\|H\|_F$,  we have  $\|\alpha I_n-A\|_2=\alpha-\omega_n$ and
$\|(\alpha I_n-\overline{A})^{-1}\|_2= (\alpha-\omega_1)^{-1}$, yielding
$\|(\alpha I_n-A)-B(\alpha I_n-\overline{A})^{-1}\overline{B}\|_2
\leq (\alpha-\omega_n)+ (\alpha-\omega_1)^{-1} \|B\|_2^2$.
We also have  
\begin{align*}
&(\alpha-\omega_n)(\alpha-\omega_1)+\|B\|_2^2-2\alpha(\alpha-\omega_1)\\
	=&-\left(\alpha+\frac{\omega_n-\omega_1}{2}\right)^2+\frac{(\omega_1+\omega_n)^2}{4}+\|B\|_2^2
	<-\left(\alpha+\frac{\omega_n-\omega_1}{2}\right)^2+\frac{\alpha^2-\|A\|_F^2}{2}, 
\end{align*}
as $\alpha^2>\|H\|_F^2=2(\|B\|_F^2+\|A\|_F^2)$.
From the fact that  $2\|A\|_F^2\geq2(\omega_1^2+\omega_n^2)\geq(\omega_n-\omega_1)^2$, we obtain 
\[
\frac{\alpha^2-\|A\|_F^2}{2}-\left(\alpha+\frac{\omega_n-\omega_1}{2}\right)^2\leq -\frac{(\alpha+\omega_n-\omega_1)^2}{2}<0.
\]
This implies $(\alpha-\omega_n)(\alpha-\omega_1)+\|B\|_2^2-2\alpha(\alpha-\omega_1)<0$.
We deduce $\|(\alpha I_n-A)-B(\alpha I_n-\overline{A})^{-1}\overline{B}\|_2<2\alpha$,
thus $2\alpha \not\in \lambda \{(\alpha I_n-A)-B(\alpha I_n-\overline{A})^{-1}\overline{B} \}$.
Therefore, $E_{\alpha}=I_n-2\alpha \left[ (\alpha I_n-A)-B(\alpha I_n-\overline{A})^{-1}\overline{B} \right]^{-1}$ 
is nonsingular.
\end{proof}

Complementing Theorem~\ref{lemma-E0}, we have   
$\lambda (E_{\alpha})$ lies outside $[0,2]$ 
when $\alpha>\|H\|_F$ 
because  
the moduli of    all eigenvalues of 
$\left[ (\alpha I_n-A)-B(\alpha I_n-\overline{A})^{-1}\overline{B} \right]^{-1}$ are   greater than $(2\alpha)^{-1}$.

\begin{theorem}\label{lemma-F0}
Assume that $\alpha> \varrho\|H\|_F+\frac{1}{2}(\varrho-1)^{-1} \|B\|_F$  with $\varrho>1$. Then $\|F_{\alpha}\|_2<1$ 
with $F_{\alpha}$ defined in \eqref{EFalpha}.
\end{theorem}

\begin{proof}
Let $\omega_1$ be  the largest eigenvalue  of $A$.
Then  it holds that  
\[
\|(\alpha I_n-A)^{-1}\|_2=(\alpha-\omega_1)^{-1}, \ \ \ 
\|F_{\alpha}\|_2 \leq \frac{2\alpha\|B\|_2}{(\alpha-\omega_1)^2-\|B\|_2^2}.
\]
We shall show that  $\|B\|_2/[(\alpha-\omega_1)^2-\|B\|_2^2]$, 
in the right-hand-side of  the inequality above, 
is bounded strictly from above by $(2\alpha)^{-1}$ when
$\alpha> \varrho\|H\|_F+\frac{1}{2} (\varrho-1)^{-1} \|B\|_F$, or equivalently
\begin{align}\label{alpha-ineq}
(\alpha-\omega_1)^2-2\alpha\|B\|_2-\|B\|_2^2>0.
\end{align}
If $\|B\|_2+\omega_1\leq0$,  \eqref{alpha-ineq} is apparently valid.
When $\|B\|_2+\omega_1>0$ and considering the left-hand-side of \eqref{alpha-ineq} as a quadratic in $\alpha$, \eqref{alpha-ineq} holds if and only if
$\alpha>\|B\|_2+\omega_1+\sqrt{2\|B\|_2 (\|B\|_2+\omega_1)}$.
With $\eta > 0$ and $\eta_1 \equiv 1/(2\eta^2)$, from the equality 
\[ \sqrt{2\|B\|_2(\|B\|_2+\omega_1)}
=\left[ \eta\sqrt{\|B\|_2}+\sqrt{\eta_1 (\|B\|_2+\omega_1)} \right]^2
-\eta^2\|B\|_2-\eta_1 \left(\|B\|_2+\omega_1 \right), 
\]
we deduce that  
\begin{align*}
&\|B\|_2+\omega_1+\sqrt{2\|B\|_2(\|B\|_2+\omega_1)}\\
=& \left[\eta\sqrt{\|B\|_2}+ \sqrt{\eta_1(\|B\|_2+\omega_1)} \right]^2+\left(1-\eta^2- \eta_1 \right)\|B\|_2 
 + \left(1-\eta_1 \right)\omega_1 \\
 \leq & \ \ \eta^2\|B\|_2+\left(1+\eta_1 \right)(\|B\|_2+\omega_1).
\end{align*}
With $\eta^2=\frac{1}{2}(\varrho-1)^{-1}$, we get   
$\eta^2\|B\|_2+\left(1+\eta_1 \right)(\|B\|_2+\omega_1)<\alpha$, thus our result.
\end{proof}

Theorem~\ref{lemma-F0} demonstrates  that when $\varrho$ is chosen as some moderate real positive scalar,
such as $\sqrt{2}$, then the corresponding lower  bound  will be a good  
candidate for the initial $\alpha$. Additionally, 
 when the condition in Theorem~\ref{lemma-F0} is satisfied, $E_{\alpha}$ 
and $I_n-F_{\alpha}\overline{F}_{\alpha}$ are nonsingular.  

Although Theorems~\ref{lemma-E0} and \ref{lemma-F0} show that a small $\alpha$ is sufficient for 
$E_{\alpha}$ and $I_n-F_{\alpha}\overline{F}_{\alpha}$ to be nonsingular, 
the minimization of $|(\lambda+\alpha)/(\lambda-\alpha)|$ for an optimal $\alpha$ 
deserves further consideration, 
for the fast convergence of the DA.  
For the optimal $\alpha$,  \cite{hlll} proposed some remarkable techniques for the suboptimal solution  
$\alpha_{opt}:=\argmin_{\alpha>0} \max_{\Re(\lambda)<0} \left|\frac{\lambda +\alpha}{\lambda -\alpha}\right|$.  
With some prior knowledge (in $\mathcal{D}$ below) of the eigenvalues of $H$, 
\cite{hlll} essentially solves the following optimization problem:
\[
	\alpha_{sopt}:=\argmin_{\alpha>0} \max_{\zeta \in \mathcal{D}}
	\left|\frac{\zeta+\alpha}{\zeta-\alpha}\right|, 
	\qquad  
	\text{where} \quad     \{\lambda\in \lambda(H): \Re(\lambda)<0\}\subset \mathcal{D}\subset \mathbb{C}_{-}. 
\]
With $\mathcal{D}$ being an interval, a disk, an ellipse or a rectangle, \cite[Theorem~2.1]{hlll} 
considers the suboptimal solution $\alpha_{sopt}$. 
The technique can be applied to  \eqref{MLalpha} 
for a suboptimal $\alpha$ when the distance between 
$\{\lambda\in \lambda(H): \Re(\lambda)<0\}$ and the imaginary axis is known.

From now on, we will always assume $\alpha>0$  such that $\alpha I_{2n}-H$, 
 $\alpha I_n-A$,   $I_n-(\alpha I_n-\overline{A})^{-1}\overline{B}(\alpha I_n-A)^{-1}B$
 and $E_{\alpha}$ are nonsingular and also assume that  $1\notin\sigma(F_{\alpha})$ 
(before the discussion in Section~3.3). 

\subsection{Algorithm}

We now construct a new symplectic pair
by applying the doubling action to a given   symplectic pair $(M, L)$ in SSF-1 in \eqref{MLalpha}; i.e., for $E^{\HH}=E$, 
$F^{\T}=F\in\mathbb{C}^{n\times n}$, we have 
\begin{align}\label{ML}
M=\begin{bmatrix}E& \mathbf 0\\& \\F&I_n\end{bmatrix}, \qquad
L=\begin{bmatrix}I_n&\overline{F}\\& \\ \mathbf 0&\overline{E}\end{bmatrix}. 
\end{align}

\begin{theorem}\label{theorem-doubling-trans1} 
For $M,L$ in \eqref{ML} with $1\notin\sigma(F)$, there exists
$[M_{*},\, L_{*}] \in \mathcal{N}(M, L)$ 
such that $(\widetilde{M},\, \widetilde{L}) = (M_{*}M, \, L_{*}L)$,
from the doubling transformation of $(M,\, L)$,
is a symplectic pair in SSF-1.   
Furthermore,   $[\widetilde{M},\, \widetilde{L}]$ retains the SSF-1:  
\[
\widetilde{M}=\begin{bmatrix}\widetilde{E}& \mathbf 0 \\&\\ \widetilde{F}& I_n\end{bmatrix},  \qquad
\widetilde{L}=\begin{bmatrix}I_n& \overline{\widetilde{F}}\\&\\ \mathbf  0& \overline{\widetilde{E}}\end{bmatrix}, 
\]
with $\widetilde{E}^{\HH}=\widetilde{E}, \widetilde{F}^{\T}=\widetilde{F}\in\mathbb{C}^{n\times n}$. 
\end{theorem}

\begin{proof}
Let
\[
M_*=\begin{bmatrix}E+E\overline{F}( I_n-F\overline{F} )^{-1}F& \mathbf 0\\ 
\\
\overline{E}(I_n- F\overline{F} )^{-1}F& I_n\end{bmatrix},\qquad
L_*=\begin{bmatrix}I_n& E\overline{F}(I_n- F\overline{F} )^{-1}
\\ \\ 
\mathbf 0& \overline{E}( I_n-F\overline{F} )^{-1}\end{bmatrix}.
\]
We have $\rank([M_{*},\, L_{*}])=2n$ and
\[
M_{*}L=\begin{bmatrix}E(I_n-\overline{F}F)^{-1}&E\overline{F}(I_n-F\overline{F})^{-1}\\ \\
\overline{E}(I_n-F\overline{F})^{-1}F&\overline{E}(I_n-F\overline{F})^{-1} \end{bmatrix}=L_{*}M,
\]
implying that $[M_{*},\, L_{*}] \in \mathcal{N}(M, L)$. 
Routine manipulations yield
\[
M_{*}M=\begin{bmatrix}E(I_n-\overline{F}F)^{-1}E&\mathbf 0\\ \\F+\overline{E}F(I_n-\overline{F}F)^{-1}E&I_n\end{bmatrix}, \qquad
L_{*}L=\begin{bmatrix}I_n&\overline{F}+E\overline{F}(I_n-F\overline{F})^{-1}\overline{E}
\\ \\ 
\mathbf 0&\overline{E}(I_n-F\overline{F})^{-1}\overline{E}\end{bmatrix}.
\]
With $\widetilde{E}=E(I_n-\overline{F}F)^{-1}E$ and
$\widetilde{F}=F+\overline{E}F(I_n-\overline{F}F)^{-1}E$,  the result follows.
\end{proof}

If we initially take $M_0=M_{\alpha}$ and $L_0=L_{\alpha}$ 
(from \eqref{MLalpha}), indicating that
$E_0=E_{\alpha}$ and $F_0=F_{\alpha}$ (specified in \eqref{EFalpha}), 
then successive doubling transformations in Theorem~\ref{theorem-doubling-trans1}
produce a sequence of symplectic pairs 
$(M_{k}, L_{k})$ provided that $(I_n-\overline{F}_kF_k)$ are nonsingular for $k\geq 0$.
Specifically, we have a well-defined doubling iteration, provided that $1 \not\in \sigma (F_k)$: (for $k=0, 1, \ldots$)
\begin{equation}\label{doubling iteration}
E_{k+1}=E_k(I_n-\overline{F}_kF_k)^{-1}E_k, \qquad
F_{k+1}=F_k+\overline{E}_kF_k(I_n-\overline{F}_kF_k)^{-1}E_k .
\end{equation}
Assuming \eqref{eigenH}  
with $S_{\alpha}\equiv(S-\alpha I_l)^{-1}(S+\alpha I_l)$, 
Lemmas~\ref{lemme-doubling} and~\ref{eigenML0} imply 
\begin{align}\label{MLX}
M_{k} \begin{bmatrix} X_1 \\ \\ X_2 \end{bmatrix}
=L_{k} \begin{bmatrix} X_1 \\ \\ X_2 \end{bmatrix} S_{\alpha}^{2^k}, \ \ \ 
M_{k}=\begin{bmatrix}E_{k}&\mathbf 0\\ \\F_{k}&I_n\end{bmatrix}, \ \  L_{k}=\begin{bmatrix}I_n&\overline{F}_k\\ \\
 \mathbf 0&\overline{E}_k\end{bmatrix}.
\end{align}

The DA in \eqref{doubling iteration} has two iterative formulae for $E_k$ and $F_k$. Interestingly, 
the SDAs for Riccati equations and quadratic palindromic eigenvalue problems~\cite{cfl,cflw,chlw} have three, 
those for nonsymmetric algebraic Riccati equations~\cite{lcjl1,lcjl2} have four, 
while the PDA for the linear palindromic eigenvalue problem~\cite{lccl} has one. 

\subsection*{Convergence}
We next consider the convergence of the DA. 
Without loss of generality, we assume for the moment  that $1\notin \sigma(F_k)$ 
for all $k=0, 1, \ldots$. For the case that $1\in \sigma(F_k)$ for some $k$, 
Theorem~\ref{theorem-bound} below essentially demonstrates that the following convergence result still hold. 
We also require the technical assumption that $X_1$ and 
$\left[
		X_1, \Psi_{11}
\right]$, respectively,  
are   nonsingular in Theorems~\ref{convergence} and \ref{thm:conv_pure} below. 

\begin{theorem}\label{convergence}
Assume that $H$ possesses no purely imaginary eigenvalue and \\ 
$H [X_1^{\T},\, X_2^{\T}\ ]^{\T}=[X_1^{\T},\, X_2^{\T}\ ]^{\T}S$
with  $X_1, X_2, S\in\mathbb{C}^{n\times n}$, where
$\lambda(S)$ is in the interior of the left half plane.    Then for
$\{E_k\}$  and $\{F_k\}$ generated by \eqref{doubling iteration}, we have
$\lim_{k\rightarrow\infty}E_k=0$ and $\lim_{k\rightarrow\infty}F_k=-X_2X_1^{-1}$,
both converging quadratically.
\end{theorem}

\begin{proof}
Let  $S_{\alpha} \equiv (S-\alpha I_n)^{-1}(S+\alpha I_n)$.  Note that the  
spectral radius of $S_{\alpha}$ is less than $1$ when $\alpha>0$.  
The proof is similar to that of \cite[Corollary~3.2]{lx06}.  
\end{proof}

The following theorem illustrates the linear convergence of the proposed DA when 
some purely imaginary eigenvalues exist. 

Let the Jordan decompositions of $J_{2p_{r,j}}(\ii\omega_j+\alpha)[J_{2p_{r,j}}(\ii\omega_j-\alpha)]^{-1}$ be\\ 
$J_{2p_{r,j}}(\ii\omega_j+\alpha)[J_{2p_{r,j}}(\ii\omega_j-\alpha)]^{-1}=Q_{r,j}J_{2p_{r,j}}(\ee^{\ii\theta_j}) Q_{r,j}^{-1}$ 
for $r=1, \cdots, l_j$ and $j=1, \cdots, q$.  
Denote $W_{\omega}=[W_{1,\omega}^{\T}, W_{2,\omega}^{\T}]^{\T}$,  
$Z_{\omega}=[Z_{1,\omega}^{\T}, Z_{2,\omega}^{\T}]^{\T}$,  
$Q_{r,j} =\begin{bmatrix}
	Q_{r,j}^{(11)}&Q_{r,j}^{(12)}\\ \\ Q_{r,j}^{(21)}&Q_{r,j}^{(22)}
\end{bmatrix}$ and, for $s',t'=1,2$,  
\begin{align*}
	Q^{(s't')}&:= \bigoplus_{j=1}^q \bigoplus_{r=1}^{l_j} Q_{r,j}^{(s't')}
	\\
	\Psi_{11} & \equiv W_{1,\omega}Q^{(11)}+Z_{1,\omega}Q^{(21)}, 
	\ \ \ \Psi_{21} \equiv W_{2,\omega}Q^{(11)}+Z_{2,\omega}Q^{(21)}. 
\end{align*}


\begin{theorem}\label{thm:conv_pure}
Assume that the partial multiplicities of $H$ associated 
with the purely  imaginary eigenvalues are all even, and $H$ has the eigen-decomposition 
specified in \eqref{eq:eigen_decomp}. 
Writing $X=[
	X_1^{\T}, X_2^{\T}
]^{\T}$,  
provided that 
$[
	X_1, \Psi_{11}
]$ is nonsingular, 
we  then  have $\lim_{k\to \infty}E_k=0$ and 
$\lim_{k\to \infty} F_k=  [X_2, \Psi_{21}][	X_1, \Psi_{11}]^{-1}$, both converging linearly.
\end{theorem}
\begin{proof}
By \eqref{eq:eigen_decomp} and Lemmas~\ref{lemme-doubling} and~\ref{eigenML0}, we have 
\begin{align}\label{eq:doubling_pure}
	M_{k}
	\begin{bmatrix}
		X_1&W_{1,\omega}& \overline{X_2} & Z_{1,\omega}\\
		X_2&W_{2,\omega}& \overline{X_1}&Z_{2,\omega}
	\end{bmatrix}
	=
	L_k
	\begin{bmatrix}
		X_1&W_{1,\omega}& \overline{X_2} & Z_{1,\omega}\\
		X_2&W_{2,\omega}& \overline{X_1}&Z_{2,\omega}
	\end{bmatrix}
	\widetilde S_{\alpha}^{2^k},
\end{align}
where 
	$\widetilde S_{\alpha}
	=
(\widetilde{S} + \alpha I)(\widetilde{S} - \alpha I)^{-1}$ with $\widetilde{S}$ from  \eqref{eq:eigen_decomp}. Let  $\Pi_{\omega}$ be the permutation matrix satisfying 
\begin{align*}
	&
	\Pi_{\omega}\diag\left\{S, -\overline{S}; \bigoplus_{j=1}^q \bigoplus_{r=1}^{l_j} J_{2p_{r,j}}(\ii \omega_j)\right\} 
	\Pi_{\omega}^{\T}
	=
	\widetilde{S}, 
\end{align*}
and denote $\mathcal{D} \equiv \diag\left\{ I_{n_1},I_{n_1}; \bigoplus_{j=1}^q \bigoplus_{r=1}^{l_j}Q_{r,j}\right\}$, 
$J_{\omega,\theta}= \bigoplus_{j=1}^q \bigoplus_{r=1}^{l_j} J_{p_{r,j}}(\ee^{\ii \theta_j})$, 
and $S_{\alpha}:=(S+\alpha I)(S-\alpha I)^{-1}$, 
it holds  that     
\begin{align*}
	\widetilde {S}_{\alpha}
	=
	\begin{multlined}[t]
		\Big(\Pi_{\omega}
			\mathcal{D}
			\Pi_{\omega}^{\T}
		\Big)
		\begin{bmatrix}
			S_{\alpha}&&&\\&J_{\omega,\theta}&&\Omega_{\omega}\\
			& &\overline{S}_{\alpha}^{-1}&\\
			&&&J_{\omega,\theta}
		\end{bmatrix}
		\Big(\Pi_{\omega}
			\mathcal{D}^{-1} 
			\Pi_{\omega}^{\T}
		\Big). 
	\end{multlined}
\end{align*}
This further  implies  
\begin{align}\label{eq:wtdS2^k}
	\widetilde {S}_{\alpha}^{2^k}
	=
	\begin{multlined}[t]
		\Big(\Pi_{\omega}
			\mathcal{D}
			\Pi_{\omega}^{\T}
		\Big)
		\begin{bmatrix}
			S_{\alpha}^{2^k}&&&\\&J_{\omega,\theta}^{2^k}&&\Omega_{\omega,\theta,k}\\
			& &\overline{S}_{\alpha}^{-2^k}&\\
			&&&J_{\omega,\theta}^{2^k}
		\end{bmatrix}
		\Big(\Pi_{\omega}
			\mathcal{D}^{-1} 
			\Pi_{\omega}^{\T}
		\Big)
	\end{multlined}
\end{align}
with $\Omega_{\omega, \theta, k}=\bigoplus_{j=1}^q \bigoplus_{r=1}^{l_j}
J_{2p_{r,j}}^{2^k}(\ee^{\ii\theta_j})(1:p_{r,j}, p_{r,j}+1:2p_{r,j})$.
By \eqref{eq:doubling_pure} and \eqref{eq:wtdS2^k} we have 
\begin{align*}
	M_{k}
	\begin{bmatrix}
		X_1&W_{1,\omega}& \overline{X_2} & Z_{1,\omega}\\
		X_2&W_{2,\omega}& \overline{X_1}&Z_{2,\omega}
	\end{bmatrix}
	\left(\Pi_{\omega} \mathcal{D} \Pi_{\omega}^{\T}\right)
	=
	\begin{multlined}[t]
		L_k
		\begin{bmatrix}
			X_1&W_{1,\omega}& \overline{X_2} & Z_{1,\omega}\\
			X_2&W_{2,\omega}& \overline{X_1}&Z_{2,\omega}
		\end{bmatrix}\left(\Pi_{\omega} \mathcal{D} \Pi_{\omega}^{\T}\right)
		\\
	\cdot\begin{bmatrix}
			S_{\alpha}^{2^k}&&&\\&J_{\omega,\theta}^{2^k}&&\Omega_{\omega,\theta,k}\\
			& &\overline{S}_{\alpha}^{-2^k}&\\
			&&&J_{\omega,\theta}^{2^k}
		\end{bmatrix}.
	\end{multlined}
\end{align*}
Similar to the proof of \cite[Theorem~4.2]{hl}, we obtain the result.  
\end{proof}

\bigskip

Next assume that we have acquired a sympletic pair $(M_{k},\, L_{k})$ with 
$\|E_k\|_F<\mathbf{u}$, where $\mathbf{u}$ is some small tolerance. The question 
is then how to compute the eigenvalues and eigenvectors of $H$ 
from $E_k$ and $F_k$.  Without loss of generality, we just show  
the details for the case that no purely imaginary eigenvalues exist. 

Denote the error $Z_k \equiv F_k+X_2X_1^{-1}$ (Theorem~\ref{convergence} and \eqref{doubling iteration} suggest  
$\|Z_k\|_F< \mathbf{u}$), where $X_1, X_2\in\mathbb{C}^{n\times n}$ satisfy
$H \left[ X_1^{\T}, \, X_2^{\T} \right]^{\T}=\left[ X_1^{\T}, \, X_2^{\T} \right]^{\T} S$ with
$\lambda(S)\subseteq\mathbb{C}_{-}$, we have
\begin{align}\label{approx-eigenvalue}
H \begin{bmatrix} \ \ I_n \\ \\ -F_k \end{bmatrix} = \begin{bmatrix} \ \ I_n \\ \\ -F_k \end{bmatrix} X_1SX_1^{-1} 
+ \begin{bmatrix}\mathbf  0 \\ \\ Z_k \end{bmatrix} X_1SX_1^{-1} -H \begin{bmatrix}\mathbf  0 \\ \\ Z_k \end{bmatrix}.
\end{align}
Pre-  and post-multiplying 
$\left[ I_n, \, -F_k^{\HH} \right]$ 
and $(I_n+F_k^{\HH}F_k)^{-1}$, respectively,  to both sides of
\eqref{approx-eigenvalue}, we obtain
\begin{align*}
&(I_n+F_k^{\HH}F_k)X_1SX_1^{-1}(I_n+F_k^{\HH}F_k)^{-1}\\
= &  \left\{ \left[ I_n, \, -F_k^{\HH} \right]  H 
\left[ I_n, \, -F_k^{\T} \right]^{\T} 
+ (F_k^{\HH}Z_kX_1SX_1^{-1}+BZ_k+F_k^{\HH}\overline{A}Z_k) \right\} 
(I_n+F_k^{\HH}F_k)^{-1}.
\end{align*}
Accordingly, we can take  the eigenvalues of  
$H_k\equiv \left[ I_n, \, -F_k^{\HH} \right] H 
\left[ I_n, \, -F_k^{\T} \right]^{\T}(I_n+F_k^{\HH}F_k)^{-1}$ to
approximate $\lambda(S)$ (the stable subspectrum of $H$). 
By the generalized Bauer-Fike theorem~\cite{ss90}, 
when the eigenvalues $\lambda_p(S)$ have 
Jordan blocks of maximum size $m$, 
there exists an eigenvalue $\lambda_q(H_k)$ such that 
\begin{align*}
	\frac{|\lambda_p(S)-\lambda_q(H_k)|^m}{[1+|\lambda_p(S)-\lambda_q(H_k)|]^{m-1}}
	&\leq
	\Upsilon\|(F_k^{\HH}Z_kX_1SX_1^{-1}+BZ_k+F_k^{\HH}\overline{A}Z_k)
(I_n+F_k^{\HH}F_k)^{-1}\|_2\\
     &\leq
\Upsilon\|F_k^{\HH}Z_kX_1SX_1^{-1}+BZ_k+F_k^{\HH}\overline{A}Z_k\|_2,
\end{align*}
for some $\Upsilon > 0$ associated with $S$.  
Consequently,  we can approximate $\lambda(S)$ 
 by $\lambda(H_k)$.

\subsection{Double-Cayley Transform}

When $1\in\sigma(F_{k_0})$ for some $k_0>1$ 
(or the condition in Theorem~\ref{theorem-doubling-trans1} is violated), 
we cannot construct the new symplectic pair  $(M_{k_0+1},\, L_{k_0+1})$ via the doubling transformation 
in \eqref{doubling iteration}. 
In this section, we divert the DA from this potential interruption using a DCT.
We shall also prove the efficiency of the technique, not requiring  a restart with a new $\alpha$.  
It is worthwhile to point that  the DCT may be applied  
when  $I - \overline{F}_{k_0} F_{k_0}$ is ill-conditioned.  
In practice, we may set a tolerance $\mathbf{u}$ and 
once the singular values  of $F_{k_0}$  
satisfy $\frac{\min_{\sigma\in \sigma(F_{k_0})}|\sigma-1|}{\max_{\sigma\in \sigma(F_{k_0})}|\sigma-1|}<\mathbf{u}$,   
the DCT  is then applied. 

We require the following results firstly.

\begin{lemma}\label{lemma-E}
Assume that the  doubling iteration  \eqref{doubling iteration} does not break off 
for all $k<k_0$.  
If $E_0$ is nonsingular, so are  $E_{k}$ $(0<k \leq k_0)$. 
\end{lemma}

\begin{proof}
This directly follows from  
$E_{k+1}=E_k(I_n-\overline{F}_kF_k)^{-1}E_k$ in \eqref{doubling iteration}. 
\end{proof}

Obviously, Lemma~\ref{lemma-E}  suggests that $M_{k_0}$ and $L_{k_0}$, defined in \eqref{MLX},
 are both nonsingular and so is
\[
L_{k_0}^{-1}M_{k_0}=\begin{bmatrix}E_{k_0}-\overline{F}_{k_0}\overline{E}_{k_0}^{-1}F_{k_0}&-\overline{F}_{k_0}\overline{E}_{k_0}^{-1}\\
&\\
\overline{E}_{k_0}^{-1}F_{k_0}&\overline{E}_{k_0}^{-1}\end{bmatrix}.
\]
Since $L_{k_0}^{-1}M_{k_0} [X_1^{\T},\, X_2^{\T}]^{\T}=
 [X_1^{\T},\, X_2^{\T}]^{\T}S_{\alpha}^{2^{k_0}}$,
the fact that  $\{0, \alpha\} \not\subset\lambda(H)$  implies 
$L_{k_0}^{-1}M_{k_0} \pm I_{2n}$  
are nonsingular. 
Consequently, we have the following  theorem.

\begin{theorem}\label{theorem-doubling-trans2}
Let $\vartheta\in \{-1, 1\}$ and  $\beta \in \mathbb{R}$. Provided that  $\vartheta \notin \lambda(E_{k_0})$, 
then   
\begin{enumerate}
  \item [{\em (a)}] $Z=\vartheta I_{n}-E_{k_0}+\vartheta \overline{F}_{k_0}(\vartheta \overline{E}_{k_0}-I_n)^{-1}F_{k_0}$ 
	  is nonsingular;  
  \item [{\em (b)}] $(\widehat{H}+\beta \vartheta I_{2n})  [X_1^{\T},\, X_2^{\T}]^{\T}
                    =(\widehat{H}-\beta \vartheta I_{2n})  [X_1^{\T},\, X_2^{\T}]^{\T} 
                    (\vartheta S_{\alpha}^{2^{k_0}})$  
                    with  
                    $\widehat{A}=\beta \vartheta I_{n}-2\beta Z^{-1}$, 
                    $\widehat{B}=(\beta I_n-\vartheta\widehat{A})\overline{F}_{k_0}(\overline{E}_{k_0}-\vartheta I_n)^{-1}$
                    and $\widehat{H}=\begin{bmatrix}\ \ \widehat{A}&\ \ \widehat{B}\\ \\-\overline{\widehat{B}}&-\overline{\widehat{A}}\end{bmatrix}$; and
  \item [{\em (c)}] $\widehat{A}$ is Hermitian and $\widehat{B}$ is symmetric.
\end{enumerate}
\end{theorem}

\begin{proof}
For (a) with $\vartheta\notin\lambda(E_{k_0})$, $E_{k_0}-\vartheta I_n$ is nonsingular and so is 
\[
K\triangleq\begin{bmatrix}I_n&\overline{F}_{k_0}(I_n-\vartheta \overline{E}_{k_0})^{-1}\\
&\\
\mathbf 0&(\overline{E}_{k_0}^{-1}-\vartheta I_n)^{-1}\end{bmatrix}. 
\]
In addition, pre-multiplying $L_{k_0}^{-1}M_{k_0}$ by $K$ gives
\[
K(L_{k_0}^{-1}M_{k_0}-\vartheta I_{2n})
=\begin{bmatrix}E_{k_0}-\vartheta I_n+\vartheta \overline{F}_{k_0}(I_n-\vartheta \overline{E}_{k_0})^{-1}F_{k_0} & \mathbf 0\\
&\\
(I_n-\vartheta \overline{E}_{k_0})^{-1}F_{k_0}&I_n\end{bmatrix},
\]
implying that $Z=\vartheta I_{n}-E_{k_0}+\vartheta \overline{F}_{k_0}(\vartheta \overline{E}_{k_0}-I_n)^{-1}F_{k_0}$ 
is nonsingular.

For (b), manipulations show that
$\widehat{H}=\beta\vartheta(L_{k_0}^{-1}M_{k_0}-\vartheta I_n)^{-1}(L_{k_0}^{-1}M_{k_0}+\vartheta I_n)$. Then 
$M_{k_0}\ [X_1^{\T},\, X_2^{\T}]^{\T}
=L_{k_0} [X_1^{\T},\, X_2^{\T}]^{\T}S_{\alpha}^{2^{k_0}}$ 
implies  
\[
(L_{k_0}^{-1}M_{k_0}-\vartheta I_n)^{-1}(L_{k_0}^{-1}M_{k_0}+\vartheta I_n) [X_1^{\T},\, X_2^{\T}]^{\T}
= [X_1^{\T},\, X_2^{\T}]^{\T}(S_{\alpha}^{2^{k_0}}-\vartheta I_n)^{-1}
(S_{\alpha}^{2^{k_0}}+\vartheta I_n),
\]
leading to 
$\widehat{H} [X_1^{\T},\, X_2^{\T}]^{\T}= [X_1^{\T},\, X_2^{\T}]^{\T}
[\beta \vartheta (S_{\alpha}^{2^{k_0}}-\vartheta I_n)^{-1}(S_{\alpha}^{2^{k_0}}+\vartheta I_n)]$. 
Consequently, the result follows from the resulting equalities
\begin{align*}
& (\widehat{H}+\beta \vartheta I_n) [X_1^{\T},\, X_2^{\T}]^{\T}
= [X_1^{\T},\, X_2^{\T}]^{\T}[2\beta \vartheta
(S_{\alpha}^{2^{k_0}}-\vartheta I_n)^{-1}S_{\alpha}^{2^{k_0}}], \\
& (\widehat{H}-\beta \vartheta I_n) [X_1^{\T},\, X_2^{\T}]^{\T}
= [X_1^{\T},\, X_2^{\T}]^{\T}[2\beta (S_{\alpha}^{2^{k_0}}-\vartheta I_n)^{-1}].
\end{align*}

For (c), $\widehat{A}^{\HH}=\widehat{A}$ directly follows from its definition and the facts that
$E_{k_0}^{\HH}=E_{k_0}$ and $F_{k_0}^{\T}=F_{k_0}$.  For the symmetry of $\widehat{B}$, observe that 
\begin{align*}
\widehat{B}
&=2\beta \vartheta Z^{-1}\overline{F}_{k_0}(\overline{E}_{k_0}-\vartheta I_n)^{-1}&\\
&=2\beta \vartheta (E_{k_0}- \vartheta I_n)^{-1}[I_n+ \vartheta \overline{F}_{k_0}
(\vartheta \overline{E}_{k_0}-I_n)^{-1}F_{k_0}(\vartheta I_n-E_{k_0})^{-1}]^{-1}
\overline{F}_{k_0}(\vartheta I_n-\overline{E}_{k_0})^{-1}&\\
&=2\beta \vartheta (E_{k_0}- \vartheta I_n)^{-1} \overline{F}_{k_0}
[I_n+\vartheta (\vartheta \overline{E}_{k_0}-I_n)^{-1}F_{k_0}(\vartheta I_n-E_{k_0})^{-1}\overline{F}_{k_0}]^{-1}
(\vartheta I_n-\overline{E}_{k_0})^{-1}&\\
&=2\beta \vartheta (E_{k_0}-\vartheta I_n)^{-1} \overline{F}_{k_0}\overline{Z}^{-1}
=2\beta \vartheta (E_{k_0}-\vartheta I_n)^{-1}\overline{F}_{k_0} Z^{-\T}=B^{\T}. 
\end{align*} 
The proof is complete. 
\end{proof}

Theorem~\ref{theorem-doubling-trans2} implies 
$\widehat{H} [X_1^{\T},\, X_2^{\T}]^{\T}
=\beta \vartheta  [X_1^{\T},\, X_2^{\T}]^{\T}
(S_{\alpha}^{2^{k_0}}+\vartheta I_l)(S_{\alpha}^{2^{k_0}}-\vartheta I_l)^{-1}$, hence
each  eigenvalue  $\lambda$
of $H$ corresponds to an eigenvalue $\mu$ of $\widehat{H}$:
\begin{align}\label{eigenvalue-mu}
\mu=f(\lambda)\triangleq
\beta\vartheta \cdot \frac{(\lambda+\alpha)^{2^{k_0}}+\vartheta(\lambda-\alpha)^{2^{k_0}}}
{(\lambda+\alpha)^{2^{k_0}}-\vartheta(\lambda-\alpha)^{2^{k_0}}}.
\end{align}
More specifically, for $\lambda\in\lambda(H)$, we have
\begin{align*}
\left\{
\begin{array}{ll}
\{\mu, \ \overline{\mu}=f(\overline{\lambda}),\  -\mu=f(-\lambda),\  -\overline{\mu}=f(-\overline{\lambda})\}\subseteq \lambda(\widehat{H}), & \quad \text{if} \ \Re(\lambda)\Im(\lambda)\neq0;\\
\{\mu,  \  -\mu=f(-\lambda) \}\subseteq \lambda(\widehat{H}), & \quad \text{if} \ \Im(\lambda)=0;\\
\{\mu,  \ \overline{\mu}=f(\overline{\lambda}) \}\subseteq \lambda(\widehat{H}), & \quad \text{if} \ \Re(\lambda)=0.\\
\end{array}
\right.
\end{align*}
In addition, $\mu \in \lambda(\widehat{H})$ is purely imaginary if $\lambda \in \lambda(H)$ is so. 
Equivalently, there exists no purely imaginary eigenvalues for $\widehat{H}$ 
 when there is none for $H$.  

Next select $\gamma\in\mathbb{R}$ with
$\gamma I_n-\widehat{A}$ and 
$I_n-(\gamma I_n-\overline{\widehat{A}})^{-1}\overline{\widehat{B}} (\gamma I_n-\widehat{A})^{-1}\widehat{B}$
being nonsingular. Theorem~\ref{theorem-SSF-1} could then be applied to $\widehat{A}$ and $\widehat{B}$, which are  defined
in Theorem~\ref{theorem-doubling-trans2}, to obtain a new SSF-1 derived from  $\widehat{H}$.
Thus, we have
\begin{multline*}
M_{k_0+1}\begin{bmatrix}X_1\\ \\ X_2\end{bmatrix}
= L_{k_0+1}\begin{bmatrix}X_1\\ \\ X_2\end{bmatrix}
	\left[ \beta\vartheta (S_{\alpha}^{2^{k_0}}+\vartheta I_l)
	(S_{\alpha}^{2^{k_0}}-\vartheta I_l)^{-1} +\gamma I_l \right]
	\\
	\cdot 
	\left[ \beta\vartheta (S_{\alpha}^{2^{k_0}}+\vartheta I_l)
	(S_{\alpha}^{2^{k_0}}-\vartheta I_l)^{-1} -\gamma I_l \right]^{-1}, 
\end{multline*}
with
\begin{align*}
&M_{k_0+1}=\begin{bmatrix}E_{{k_0}+1}&\mathbf 0\\ \\ F_{k_0+1}&I_n\end{bmatrix}, \qquad
L_{k_0+1}=\begin{bmatrix}I_n&\overline{F}_{k_0+1}\\ \\ \mathbf 0&\overline{E}_{k_0+1}\end{bmatrix},\\
	&E_{k_0+1}=I_n-2\gamma\left[(\gamma I_n-\widehat{A})
-\widehat{B}(\gamma I_n-\overline{\widehat{A}})^{-1}\overline{\widehat{B}}\right]^{-1},\\
&F_{k_0+1}=-2\gamma(\gamma I_n-\overline{\widehat{A}})^{-1}\overline{\widehat{B}}
\left[(\gamma I_n-\widehat{A})-\widehat{B} (\gamma I_n-\overline{\widehat{A}})^{-1}\overline{\widehat{B}}\right]^{-1}.
\end{align*}
We call the above transform from $(M_{k_0}, L_{k_0})$ to 
$(M_{k_0+1}, L_{k_0+1})$, both symplectic, a DCT. 
Accordingly, with $\delta_{\lambda}\triangleq (\lambda+\alpha)(\lambda-\alpha)^{-1}$, $|\delta_{\lambda}|<1$ 
and $\varpi \triangleq (\beta-\vartheta\gamma)(\beta\vartheta +\gamma)^{-1}$, 
an eigenvalue $\mu$ of $\widehat{H}$ (in \eqref{eigenvalue-mu}) would be transformed into 
an eigenvalue $\nu$ of $(M_{k_0+1}, L_{k_0+1})$ via the following formula: (for $\lambda \in\lambda(H)$)
\begin{align*}
	\nu &\equiv\nu(\mu)=\frac{\mu+\gamma}{\mu-\gamma}\\
	&=
	\frac{\beta\vartheta[(\lambda+\alpha)^{2^{k_0}}+\vartheta(\lambda-\alpha)^{2^{k_0}}] 
	+ \gamma[(\lambda+\alpha)^{2^{k_0}}-\vartheta(\lambda-\alpha)^{2^{k_0}}]}
	{\beta\vartheta[(\lambda+\alpha)^{2^{k_0}}+\vartheta(\lambda-\alpha)^{2^{k_0}}] 
	- \gamma[(\lambda+\alpha)^{2^{k_0}}-\vartheta(\lambda-\alpha)^{2^{k_0}}]} 
	=\vartheta \cdot \frac{\varpi +\delta_{\lambda}^{2^{k_0}}}{1+\varpi \delta_{\lambda}^{2^{k_0}}}. 
\end{align*}

One may consider the condition number of $I_n-\overline{F}_{k_0+1}F_{k_0+1}$, or equivalently, 
the difference  between $1$ and $\sigma(F_{k_0+1})$. Obviously, 
$\sigma(F_{k_0})$ depends on $\gamma$. 
Without loss of generality we assume $\vartheta=1$, then with 
$\gamma=\beta(\kappa^{2^{k_0}}+1)(\kappa^{2^{k_0}}-1)^{-1}$(with $\kappa$ to be specified), we have    
\begin{align*}
	&F_{k_0+1}=
	-\frac{\kappa^{2^{k_0}}+1}{\kappa^{2^{k_0}}-1}
	\left(\frac{\overline Z}{\kappa^{2^{k_0}}-1} +I_n \right)^{-1} F_{k_0}(E_{k_0}-I_n)^{-1}
	\\
	& \ \ \ \cdot \left[\left(\frac{Z}{\kappa^{2^{k_0}}-1} +I_n \right) 
	-  \overline F_{k_0}(\overline E_{k_0}-I_n)^{-1}
\left(\frac{\overline Z}{\kappa^{2^{k_0}}-1} +I_n \right)^{-1} F_{k_0}(E_{k_0}-I_n)^{-1}\right]^{-1}Z.
\end{align*} 
Thus we can choose some $\kappa$ to make 
$I_n-\overline{F}_{k_0+1}F_{k_0+1}$  well conditioned.  
We leave the issue of an optimal  $\kappa$ or $\gamma$ for the future,   
while making random choices in our numerical experiments.    
Theorem~\ref{theorem-bound} and Corollary~\ref{corollary-bound}  below illustrate that 
$\kappa$ characterizes the convergence rate and does not have to be large.

With $\gamma>0$ and   
$\Re(\mu)<0$, we have $|\nu(\mu)|<1$. The following 
lemma reveals  more.  

\begin{lemma}\label{lemma-nu}
Provided that $\vartheta\beta, \gamma>0$, then each $\nu$ corresponding to a non-purely imaginary
eigenvalue $\lambda\in \lambda(H)$ with $\Re(\lambda)<0$ satisfies $|\nu|<1$.
\end{lemma}
\begin{proof}
Let $\xi+\ii\eta=\varrho=\delta_{\lambda}^{2^{k_0}}$, 
we then have $|\varrho|=|\delta_{\lambda}|^{2^{k_0}}$ and $|\xi|\leq |\delta_{\lambda}|^{2^{k_0}}$. 
Consequently, from the definition of $\nu$ we deduce that 
\begin{align}
|\nu|^2&=\frac{(\xi^2+\eta^2)(\beta\vartheta+\gamma)^2+(\beta-\vartheta\gamma)^2+2\vartheta\xi(\beta^2-\gamma^2)}
{(\beta\vartheta+\gamma)^2+(\beta-\vartheta\gamma)^2(\xi^2+\eta^2)+2\vartheta\xi(\beta^2-\gamma^2)} \nonumber \\
&=
\frac{|\delta_{\lambda}|^{2^{k_0+1}}+2\xi\varpi 
+\varpi^2}
{|\delta_{\lambda}|^{2^{k_0+1}}\varpi^2
+2\xi \varpi +1}. \label{xi}
\end{align}
Since $\vartheta\beta, \gamma>0$ and  the function defined in \eqref{xi}
is (i) monotone nondecreasing with respect to $\xi$  when $\beta>\vartheta\gamma$ 
or (ii) monotone non-increasing otherwise, we obtain 
\[
|\nu|^2\leq\left\{
\begin{array}{ll}
&\frac{|\delta_{\lambda}|^{2^{k_0}}(|\delta_{\lambda}|^{2^{k_0}}+2\varpi )
+\varpi^2}
{|\delta_{\lambda}|^{2^{{k_0}}}(2\varpi+|\delta_{\lambda}|^{2^{k_0}}
\varpi^2 )+1}, \quad \quad \text{if} \quad \beta>\vartheta\gamma;\\
&\\
&\frac{|\delta_{\lambda}|^{2^{k_0}}(|\delta_{\lambda}|^{2^{k_0}}-2\varpi )
+\varpi^2}
{|\delta_{\lambda}|^{2^{{k_0}}}(-2\varpi+|\delta_{\lambda}|^{2^{k_0}}
\varpi^2 )+1}, \quad \ \text{if} \quad \beta<\vartheta\gamma;\\
\end{array}
\right.
\]
which is equivalent to
\[
|\nu|^2\leq
\frac{|\delta_{\lambda}|^{2^{k_0}}(|\delta_{\lambda}|^{2^{k_0}}+2|\varpi|)+
\varpi^2}
{|\delta_{\lambda}|^{2^{{k_0}}}(2|\varpi|+
|\delta_{\lambda}|^{2^{k_0}}\varpi^2 )+1}
= \left( \frac{|\delta_{\lambda}|^{2^{k_0}}+|\varpi|}
{|\delta_{\lambda}|^{2^{k_0}}|\varpi| +1} \right)^2.
\]
Obviously, 
$(|\delta_{\lambda}|^{2^{k_0}}+|\varpi|)
(|\delta_{\lambda}|^{2^{k_0}}|\varpi| + 1)^{-1} <1$ 
from 
$|\varpi| = |\beta-\vartheta\gamma|/(\vartheta\beta+\gamma)<1$ and $|\delta_{\lambda}|<1$, thus the result follows.
\end{proof}

Lemma~\ref{lemma-nu} demonstrates  that for  $\lambda\in\lambda(H)$ satisfying $\Im(\lambda)\neq0$,
the DCT maps half of these $\lambda$ to some values inside of the unit circle and 
the other half outside. Next we consider 
the detailed relationship between $\nu$ and $\varrho=\delta_\lambda^{2^{k_0}}$,
which is vital for the convergence of the DA coupled with   
the DCT.

Obviously, when  $\vartheta\beta, \gamma>0$, we have $|\varpi|<1$.
Taking  $\gamma=\beta(\kappa^{2^{k_0}}+\vartheta)(\vartheta\kappa^{2^{k_0}}-1)^{-1} >0$ with  $\kappa>1$, 
we obtain  
$\varpi 
=-\kappa^{-2^{k_0}}$ and 
\[
\nu=\vartheta \cdot \frac{\delta_\lambda^{2^{k_0-1}}- \kappa^{-2^{k_0-1}}}
{1-\delta_\lambda^{2^{k_0-1}} \kappa^{-2^{k_0-1}}} \cdot 
\frac{\delta_\lambda^{2^{k_0-1}}+ \kappa^{-2^{k_0-1}}}
{1+\delta_\lambda^{2^{k_0-1}} \kappa^{-2^{k_0-1}}}.
\]
Denote  $\xi+\ii\eta=\delta_\lambda^{2^{k_0-1}}$ and define
\begin{eqnarray*}
\phi &=& \arctanh \delta_\lambda^{2^{k_0-1}}\\
&=& \frac{1}{2}
\ln\left|\frac{(\lambda-\alpha)^{2^{k_0-1}}+(\lambda+\alpha)^{2^{k_0-1}}}{(\lambda-\alpha)^{2^{k_0-1}}-(\lambda+\alpha)^{2^{k_0-1}}}\right|
+\frac{\ii}{2}\arg\left[\frac{(\lambda-\alpha)^{2^{k_0-1}}+(\lambda+\alpha)^{2^{k_0-1}}}
{(\lambda-\alpha)^{2^{k_0-1}}-(\lambda+\alpha)^{2^{k_0-1}}}\right],\\
\psi&=&\arctanh \kappa^{-2^{k_0-1}} 
=\frac{1}{2}\left[ \ln(1+\sqrt{|\varpi|})
-\ln(1-\sqrt{|\varpi|})\right]. 
\end{eqnarray*}
We deduce  that 
\[
\arg\left[\frac{(\lambda-\alpha)^{2^{k_0-1}}+(\lambda+\alpha)^{2^{k_0-1}}}{(\lambda-\alpha)^{2^{k_0-1}}-(\lambda+\alpha)^{2^{k_0-1}}}\right]
=\arctan \frac{2\eta}{1-\xi^2-\eta^2}\in\left(-\frac{\pi}{2}, \ \frac{\pi}{2}\right).
\]
Specifically, $\arg\left[\dfrac{(\lambda-\alpha)^{2^{k_0-1}}+(\lambda+\alpha)^{2^{k_0-1}}}{(\lambda-\alpha)^{2^{k_0-1}}-(\lambda+\alpha)^{2^{k_0-1}}}\right]=0$ when $\lambda\in\mathbb{R}$.
Moreover, by the definitions of $\phi$ and $\psi$, routine manipulations show that
\[
\nu=\vartheta\tanh(\phi-\psi)\tanh(\phi+\psi)
\]
with
\[ 
\phi \pm \psi=\frac{1}{2}\ln\left[\frac{\sqrt{\gamma+\vartheta\beta} \pm \sqrt{\vartheta\gamma-\beta}}
{\sqrt{\gamma+\vartheta\beta} \mp \sqrt{\vartheta\gamma-\beta}}
\sqrt{\frac{(1+\xi)^2+\eta^2}{(1-\xi)^2+\eta^2}}\right]
+\frac{\ii}{2}\arctan \frac{2\eta}{1-\xi^2-\eta^2}.
\]
Under the assumptions in Lemma~\ref{lemma-nu}, the following theorem  gives  a 
sharp bound for those $|\nu|$ corresponding to $\lambda$ which satisfies 
$\Im(\lambda)\neq0$ and $|\delta_\lambda| < 1$.

\begin{theorem}\label{theorem-bound}
Assume that  $\lambda$ is not a purely imaginary eigenvalue of $H$, 
$\vartheta\beta>0$ and $\kappa\geq2$.   
Then we have 
$|\nu| \leq \max\left\{|\delta_\lambda|^{2^{k_0-2}}, \
\kappa^{-2^{k_0-2}}\right\}$.
\end{theorem}
\begin{proof}
With $\gamma=\beta\frac{\kappa^{2^{k_0}}+\vartheta}{\vartheta\kappa^{2^{k_0}}-1}$ and $\cos\left(\arctan \frac{2\eta}{1-\xi^2-\eta^2}\right)>0$, we have 
\begin{align*}
\left\{\begin{array}{l}
\ln\left(\frac{\sqrt{\gamma+\vartheta\beta}+\sqrt{\vartheta\gamma-\beta}}{\sqrt{\gamma+\vartheta\beta}-\sqrt{\vartheta\gamma-\beta}}
\sqrt{\frac{(1+\xi)^2+\eta^2}{(1-\xi)^2+\eta^2}}\right)\geq0, \qquad \text{if} \quad
\frac{(1+\xi)^2+\eta^2}{(1-\xi)^2+\eta^2}\geq1;\\
\\
\ln\left(\frac{\sqrt{\gamma+\vartheta\beta}-\sqrt{\vartheta\gamma-\beta}}{\sqrt{\gamma+\vartheta\beta}+\sqrt{\vartheta\gamma-\beta}}
\sqrt{\frac{(1+\xi)^2+\eta^2}{(1-\xi)^2+\eta^2}}\right)<0,  \qquad \text{otherwise}.
\end{array}
\right.
\end{align*}
From Lemma~\ref{lemma-tanh} and  
$[(1+\xi)^2+\eta^2][(1-\xi)^2+\eta^2]^{-1}\geq1 \Leftrightarrow \xi\geq 0$, we obtain 
\[
|\nu|<
\left\{\begin{array}{l}
|\tanh(\phi-\psi)|, \qquad \text{if} \quad \xi>0;\\
|\tanh(\phi+\psi)|, \qquad \text{if} \quad \xi<0.
\end{array}
\right.
\]

Now assume that $\xi>0$ and we consider two distinct cases. 

(i) When
\[
\sqrt{\frac{(1-\xi)^2+\eta^2}{(1+\xi)^2+\eta^2}}\leq
\frac{\sqrt{\gamma+\vartheta\beta}-\sqrt{\vartheta\gamma-\beta}}{\sqrt{\gamma+\vartheta\beta}+\sqrt{\vartheta\gamma-\beta}}
\sqrt{\frac{(1+\xi)^2+\eta^2}{(1-\xi)^2+\eta^2}}<1
\]
or
\[
\frac{\sqrt{\gamma+\vartheta\beta}-\sqrt{\vartheta\gamma-\beta}}{\sqrt{\gamma+\vartheta\beta}+\sqrt{\vartheta\gamma-\beta}}
\sqrt{\frac{(1+\xi)^2+\eta^2}{(1-\xi)^2+\eta^2}}\geq1,
\]
we have 
\[
\ln\left[\sqrt{\frac{(1-\xi)^2+\eta^2}{(1+\xi)^2+\eta^2}}\right]\leq\ln\left[\frac{\sqrt{\gamma+\vartheta\beta}-\sqrt{\vartheta\gamma-\beta}}
{\sqrt{\gamma+\vartheta\beta}+\sqrt{\vartheta\gamma-\beta}}
\sqrt{\frac{(1+\xi)^2+\eta^2}{(1-\xi)^2+\eta^2}}\right]<0
\]
or
\[
0<\ln\left[\frac{\sqrt{\gamma+\vartheta\beta}-\sqrt{\vartheta\gamma-\beta}}
{\sqrt{\gamma+\vartheta\beta}+\sqrt{\vartheta\gamma-\beta}}
\sqrt{\frac{(1+\xi)^2+\eta^2}{(1-\xi)^2+\eta^2}}\right]<
\ln\left[\sqrt{\frac{(1+\xi)^2+\eta^2}{(1-\xi)^2+\eta^2}}\right].
\]
Hence by (c) and (b) in Lemma~\ref{lemma-tanh}, it is apparent that 
\begin{align*}
|\nu|^2& < |\tanh(\phi-\psi)|^2\\
&\leq 
\left| \tanh\left\{ 
\frac{1}{2}\ln\left[\sqrt{\frac{(1+\xi)^2+\eta^2}{(1-\xi)^2+\eta^2}}\right]+\frac{\ii}{2}\arctan \frac{2\eta}{1-\xi^2-\eta^2} 
\right\} \right|^2\\
&=|\tanh(\phi)|^2 = \left| \delta_\lambda \right|^{2^{k_0}},
\end{align*}
implying that $|\nu|<\left| \delta_\lambda \right|^{2^{k_0-1}}$.

(ii) When
\[
\frac{\sqrt{\gamma+\vartheta\beta}-\sqrt{\vartheta\gamma-\beta}}{\sqrt{\gamma+\vartheta\beta}+\sqrt{\vartheta\gamma-\beta}}
\sqrt{\frac{(1+\xi)^2+\eta^2}{(1-\xi)^2+\eta^2}}<\sqrt{\frac{(1-\xi)^2+\eta^2}{(1+\xi)^2+\eta^2}} <1,
\]
we define $\widehat{\xi}+\ii\widehat{\eta}=\delta_\lambda^{2^{k_0-2}}$
and without loss of generality assume that $\widehat{\xi}>0$, which satisfies $\widehat{\xi}>|\widehat{\eta}|$ for $0<\xi=\widehat{\xi}^2-\widehat{\eta}^2$.
Similar to (i),  we obtain
\[
|\nu|<|\tanh(\phi-\psi)|=
|\tanh(\widehat{\phi}-\widehat{\psi})\tanh(\widehat{\phi}+\widehat{\psi})|
<|\tanh(\widehat{\phi}-\widehat{\psi})|,
\]
 where
$\widehat{\phi}=\arctanh \delta_\lambda^{2^{k_0-2}}$ and
$\widehat{\psi}=\arctanh \kappa^{-2^{k_0-2}}$.
Since $\xi=\widehat{\xi}^2-\widehat{\eta}^2>0$ and $|\widehat{\xi}|^2+|\widehat{\eta}|^2=\left| \delta_\lambda\right|^{2^{k_0-1}}$,
we have $\widehat{\xi}^2>\frac{1}{2}|\delta_{\lambda}|^{2^{k_0-1}}$,
leading to 
\begin{eqnarray*}
|\nu|^2&<&|\tanh(\widehat{\phi}-\widehat{\psi})|^2\\
&=&
\dfrac{ \dfrac{\kappa^{2^{k_0-2}}-1}{\kappa^{2^{k_0-2}}+1} \cdot 
\dfrac{1+|\delta_{\lambda}|^{2^{k_0-1}}+2\widehat{\xi}}{1-|\delta_{\lambda}|^{2^{k_0-1}}}
+ \dfrac{\kappa^{2^{k_0-2}}+1}{\kappa^{2^{k_0-2}}-1} \cdot
\dfrac{1+|\delta_{\lambda}|^{2^{k_0-1}}-2\widehat{\xi}}{1-|\delta_{\lambda}|^{2^{k_0-1}}}-2 }
{ \dfrac{\kappa^{2^{k_0-2}}-1}{\kappa^{2^{k_0-2}}+1} \cdot
\dfrac{1+|\delta_{\lambda}|^{2^{k_0-1}}+2\widehat{\xi}}{1-|\delta_{\lambda}|^{2^{k_0-1}}}
+ \dfrac{\kappa^{2^{k_0-2}}+1}{\kappa^{2^{k_0-2}}-1} \cdot
\dfrac{1+|\delta_{\lambda}|^{2^{k_0-1}}-2\widehat{\xi}}{1-|\delta_{\lambda}|^{2^{k_0-1}}}+2}.
\end{eqnarray*}
Since $|\tanh(\widehat{\phi}-\widehat{\psi})|^2$ is  monotonically  nonincreasing with respect to $\widehat{\xi}$, taking $\widehat{\xi}=\frac{1}{\sqrt{2}}|\delta_{\lambda}|^{2^{k_0-2}}$ in the above formula yields 
\begin{align}
|\nu|^2&<|\tanh(\widehat{\phi}-\widehat{\psi})|^2 
<\dfrac{1+|\delta_{\lambda}|^{2^{k_0-1}}\kappa^{2^{k_0-1}}-\sqrt{2}\kappa^{2^{k_0-2}}|\delta_{\lambda}|^{2^{k_0-2}}}
{\kappa^{2^{k_0-1}}+|\delta_{\lambda}|^{2^{k_0-1}}-\sqrt{2}\kappa^{2^{k_0-2}}|\delta_{\lambda}|^{2^{k_0-2}}} \notag\\
&=\kappa^{-2^{k_0-1}} \cdot \left[ \dfrac{ (2^{-1/2}\kappa^{2^{k_0-2}}|\delta_{\lambda}|^{2^{k_0-2}}-1)^2+ 
2^{-1} \kappa^{2^{k_0-1}}|\delta_{\lambda}|^{2^{k_0-1}} }
{(2^{-1/2}\kappa^{-2^{k_0-2}}|\delta_{\lambda}|^{2^{k_0-2}}-1)^2+ 
2^{-1} \kappa^{-2^{k_0-1}}|\delta_{\lambda}|^{2^{k_0-1}} } \right] \label{kappa}\\
&=|\delta_{\lambda}|^{2^{k_0-1}} \cdot \left[ \dfrac{(\kappa^{2^{k_0-2}}- 2^{-1/2} |\delta_{\lambda}|^{-2^{k_0-2}})^2 
+ 2^{-1} |\delta_{\lambda}|^{-2^{k_0-1}}}
{(\kappa^{2^{k_0-2}}- 2^{-1/2} |\delta_{\lambda}|^{2^{k_0-2}} )^2  + 2^{-1} |\delta_{\lambda}|^{2^{k_0-1}}} \right]. \label{delta}
\end{align}

Obviously for $\kappa\geq2$, we obtain $( \kappa^{-1} |\delta_{\lambda}| )^{2^{k_0-2}}<1/2$.
Hence, by Lemma~\ref{lemma-distance}, when either 
\begin{description}
\item (a) $2^{-1/2} \left(|\delta_{\lambda}|\kappa\right)^{2^{k_0-2}}\leq\frac{1}{2}$, i.e., 
$\left(|\delta_{\lambda}|\kappa\right)^{2^{k_0-2}} \leq 1/\sqrt{2}$; or
\item (b) $\frac{1}{2}< 2^{-1/2} \left(|\delta_{\lambda}|\kappa\right)^{2^{k_0-2}}
\leq 1- 2^{-1/2} |\delta_{\lambda}|^{2^{k_0-2}} \kappa^{-2^{k_0-2}}$, i.e.,  
\[
\left(|\delta_{\lambda}|\kappa\right)^{2^{k_0-2}}\geq 1/\sqrt{2}, \qquad 
|\delta_{\lambda}|^{2^{k_0-2}}(\kappa^{2^{k_0-2}}+\kappa^{-2^{k_0-2}})\leq\sqrt{2}, 
\]
\end{description}
the quantity in the square brackets in \eqref{kappa} would be no greater than $1$. This indicates that  
$|\nu|^2 \leq \kappa^{-2^{k_0-1}}$ or 
$|\nu|<\kappa^{-2^{k_0-2}}$. 

When
\[
\left(|\delta_{\lambda}|\kappa\right)^{2^{k_0-2}}\geq 1/\sqrt{2}, \qquad 
|\delta_{\lambda}|^{2^{k_0-2}} (\kappa^{2^{k_0-2}}+\kappa^{-2^{k_0-2}})>\sqrt{2},
\]
which imply $|\delta_{\lambda}|^{2^{k_0-2}}>\sqrt{2}/(\kappa^{2^{k_0-2}}+ \kappa^{-2^{k_0-2}})$,  we obtain  
\begin{align}\label{delta2}
|\delta_{\lambda}|^{2^{k_0-2}}+|\delta_{\lambda}|^{-2^{k_0-2}}<\frac{\sqrt{2}\kappa^{2^{k_0-2}}}{\kappa^{2^{k_0-1}}+1}+
\frac{\kappa^{2^{k_0-1}}+1}{\sqrt{2}\kappa^{2^{k_0-2}}}<\sqrt{2}\kappa^{2^{k_0-2}},
\end{align}
where the first ``$<$'' follows from the fact that  the function 
$f(x)=x + x^{-1}$ is  monotonically decreasing when $x<1$. 
Thus, the  assumption $\kappa\geq2$ and \eqref{delta2} together affirm that
$2^{-1/2}|\delta_{\lambda}|^{2^{k_0-2}} < 2^{-1} \kappa^{2^{k_0-2}}$ and
$2^{-1/2} |\delta_{\lambda}|^{-2^{k_0-2}}\leq\kappa^{2^{k_0-2}}- 2^{-1/2} |\delta_{\lambda}|^{2^{k_0-2}}$.
Again using Lemma~\ref{lemma-distance}, we know that the quantity in the square brackets in \eqref{delta}  
is no  greater than $1$, suggesting that the value of the right-hand-side of  
\eqref{delta} will be no greater than $|\delta_{\lambda}|^{2^{k_0-1}}$,
or equivalently  $|\nu|<|\delta_{\lambda}|^{2^{k_0-2}}$.

Consequently, the result holds for the case when $\xi>0$.
The  $\xi<0$ case can be proved similarly and we omit the details.
\end{proof}

For a real $\lambda\in\lambda(H)$, we can obtain a better result, 
with the power $2^{k_0-2}$ replaced by $2^{k_0}$ in the following corollary.  

\begin{corollary}\label{corollary-bound}
Let $\kappa>1$ and $\vartheta\beta, \alpha>0$,
then for $\lambda<0$ $(\lambda \in \lambda(H))$, we have
$|\nu|\leq \max \left\{|\delta_\lambda|^{2^{k_0}}, \
\kappa^{-2^{k_0}}\right\}$.
\end{corollary}

\begin{proof}
Let $\phi \equiv \arctanh \delta_\lambda^{2^{k_0}}$,
then  $\phi=\frac{1}{2}
\ln\left[\frac{(\lambda-\alpha)^{2^{k_0}}+(\lambda+\alpha)^{2^{k_0}}}{(\lambda-\alpha)^{2^{k_0}}-(\lambda+\alpha)^{2^{k_0}}}\right]>0$
since $\lambda<0$, and $\psi \equiv \arctanh(-\kappa^{-2^{k_0}})
=-\frac{1}{2}\ln\left(\frac{\kappa^{2^{k_0}}+1}{\kappa^{2^{k_0}}-1}\right)<0$.
From the definition of $\nu$, we have $\nu=\vartheta\tanh(\phi+\psi)$.
Because  $\tanh(\omega)=(\ee^{\omega}-\ee^{-\omega})(\ee^{\omega}+\ee^{-\omega})^{-1}$, 
$\tanh(-\omega)=-\tanh(\omega)$ and $\tanh(\omega)$ is
nondecreasing with  respect to $\omega\in\mathbb{R}$, 
then when $\phi\geq|\psi|$ we have $0\leq|\nu|=\tanh(\phi+\psi)\leq\tanh(\phi)$. 
Otherwise for   $\phi<|\psi|$, we have $|\nu|=\tanh(-\psi-\phi)<\tanh(-\psi)=\kappa^{-2^{k_0}}$.
Hence, the result holds.
\end{proof}

To sum up, we propose the DCT to avoid  
the potential interruption of the DA caused by $1\in\sigma(F_{k_0})$ for some 
$k_0$. We have conducted a detailed analysis on  
the eigenvalue $\nu$ of the new pair $(M_{k_0+1}, L_{k_0+1})$,  
produces a sharp bound of $|\nu|$ in Theorem~\ref{theorem-bound} relative to $|\delta_\lambda|^{2^{k_0-2}}$. 
Furthermore, Theorem~\ref{theorem-bound} and Corollary~\ref{corollary-bound} imply that a double-Cayley step reverses the convergence \emph{at worst by two steps in general and not at all when $\lambda$ is real}. This 
guarantees the convergence of the DA when the DCT  is only occasionally called for. 
Similar comments apply when there exist some singular value $\sigma\in \sigma (F_{k_0})$  
close to unity, meaning  $I - \overline{F}_{k_0} F_{k_0}$ is ill-conditioned, 
and the double-Cayley remedy is applied. 

Note that the DCT is applicable when $\vartheta \notin \lambda(E_{k_0})$ with 
$\vartheta \in \{-1, 1\}$. In the rare occasions when the  
condition is violated, the three-recursion 
remedy proposed in subsection~3.4 will be employed.

We construct an example  to show  the need for 
the DCT. 
\begin{Example}\rm
Let $A=A^{\HH}, B=B^{\T} \in\mathbb{C}^{5\times 5}$ with 
\begin{align*}
	A=\begin{bsmallmatrix*}[r]
   0.6607           &  0.1299 - 0.1365\ii  &  0.0632 - 0.0086\ii & -0.0341 - 0.0517\ii & -0.0628 - 0.0044\ii\\
   0.1299 + 0.1365\ii &  0.2441            & -0.1293 - 0.1035\ii & -0.0363 + 0.1567\ii &  0.1042 + 0.1260\ii\\
   0.0632 + 0.0086\ii & -0.1293 + 0.1035\ii  &  0.6772           &  0.0236 + 0.0491\ii &  0.0542 + 0.0113\ii\\
  -0.0341 + 0.0517\ii & -0.0363 - 0.1567\ii  &  0.0236 - 0.0491\ii &  0.6804           & -0.0326 + 0.0427\ii\\
  -0.0628 + 0.0044\ii &  0.1042 - 0.1260\ii  &  0.0542 - 0.0113\ii & -0.0326 - 0.0427\ii &  0.6787
\end{bsmallmatrix*},
\\
B=\begin{bsmallmatrix*}[r]
  -0.5704 + 0.2984\ii  & -0.4605 - 0.0324\ii &  0.1693 - 0.3006\ii & -0.1181 + 0.4597\ii &  0.2109 + 0.0879\ii\\
  -0.4605 - 0.0324\ii  &  0.0573 - 0.1759\ii & -0.1520 + 0.0419\ii & -0.1526 - 0.0408\ii &  0.1452 - 0.2288\ii\\
   0.1693 - 0.3006\ii  & -0.1520 + 0.0419\ii &  0.4908 - 0.7534\ii &  0.1880 - 0.0406\ii & -0.1733 - 0.1743\ii\\
  -0.1181 + 0.4597\ii  & -0.1526 - 0.0408\ii &  0.1880 - 0.0406\ii & -0.1783 - 0.6552\ii & -0.5212 + 0.1871\ii\\
   0.2109 + 0.0879\ii  &  0.1452 - 0.2288\ii & -0.1733 - 0.1743\ii & -0.5212 + 0.1871\ii & -0.2548 - 0.7032\ii
\end{bsmallmatrix*}.
\end{align*}
By setting $\alpha=1$ and with the formulae in Theorem~\ref{theorem-SSF-1}, we 
have  $E_0=E_{\alpha}$ and $F_0=F_{\alpha}$:
\begin{align*}
	E_0=\begin{bsmallmatrix*}[r]
   1.2482           &  0.4505 - 0.4735\ii &  0.2193 - 0.0298\ii & -0.1182 - 0.1794\ii & -0.2179 - 0.0152\ii\\
   0.4505 + 0.4735\ii & -0.1966           & -0.4485 - 0.3591\ii & -0.1259 + 0.5435\ii &  0.3613 + 0.4371\ii\\
   0.2193 + 0.0298\ii & -0.4485 + 0.3591\ii &  1.3055           &  0.0817 + 0.1703\ii &  0.1880 + 0.0391\ii\\
  -0.1182 + 0.1794\ii & -0.1259 - 0.5435\ii &  0.0817 - 0.1703\ii &  1.3166           & -0.1132 + 0.1482\ii\\
  -0.2179 + 0.0152\ii &  0.3613 - 0.4371\ii &  0.1880 - 0.0391\ii & -0.1132 - 0.1482\ii &  1.3105
\end{bsmallmatrix*},
\\
F_0=\begin{bsmallmatrix*}[r]
  -1.0682 - 0.5623\ii & -0.8603 + 0.0680\ii  &  0.3168 + 0.5662\ii & -0.2188 - 0.8623\ii  &  0.3967 - 0.1673\ii\\
  -0.8603 + 0.0680\ii &  0.0883 + 0.3226\ii  & -0.2885 - 0.0846\ii & -0.2898 + 0.0820\ii  &  0.2745 + 0.4354\ii\\
   0.3168 + 0.5662\ii & -0.2885 - 0.0846\ii  &  0.9207 + 1.4103\ii &  0.3503 + 0.0768\ii  & -0.3258 + 0.3290\ii\\
  -0.2188 - 0.8623\ii & -0.2898 + 0.0820\ii  &  0.3503 + 0.0768\ii & -0.3329 + 1.2301\ii  & -0.9749 - 0.3510\ii\\
   0.3967 - 0.1673\ii &  0.2745 + 0.4354\ii  & -0.3258 + 0.3290\ii & -0.9749 - 0.3510\ii  & -0.4766 + 1.3165\ii
\end{bsmallmatrix*}.
\end{align*}
Applying the DA to $E_0$   and $F_0$ for $5$ iterations, 
we obtain: 
\begin{align*}
	E_5=\begin{bsmallmatrix*}[r]
   1.5012           & -0.0992 + 0.1043\ii & -0.0483 + 0.0066\ii &   0.0260 + 0.0395\ii    &  0.0480 + 0.0034\ii\\
  -0.0992 - 0.1043\ii &  1.8195           &  0.0988 + 0.0791\ii &   0.0277 - 0.1197\ii    & -0.0796 - 0.0963\ii\\
  -0.0483 - 0.0066\ii &  0.0988 - 0.0791\ii &  1.4886           &  -0.0180 - 0.0375\ii    & -0.0414 - 0.0086\ii\\
   0.0260 - 0.0395\ii &  0.0277 + 0.1197\ii & -0.0180 + 0.0375\ii &   1.4861              &  0.0249 - 0.0326\ii\\
   0.0480 - 0.0034\ii & -0.0796 + 0.0963\ii & -0.0414 + 0.0086\ii &   0.0249 + 0.0326\ii    &  1.4875
   \end{bsmallmatrix*},
\\
F_5=\begin{bsmallmatrix*}[r]
 -0.9956 - 0.6352\ii &  -0.7338 + 0.3015\ii  &  0.2834 + 0.6319\ii & -0.1291 - 0.8499\ii &   0.4238 - 0.2379\ii\\
 -0.7338 + 0.3015\ii &  -0.5359 + 0.0786\ii  & -0.3942 - 0.2753\ii & -0.4018 + 0.2612\ii &   0.3380 + 0.6318\ii\\
  0.2834 + 0.6319\ii &  -0.3942 - 0.2753\ii  &  0.9025 + 1.2909\ii &  0.2689 + 0.0968\ii &  -0.3410 + 0.3895\ii\\
 -0.1291 - 0.8499\ii &  -0.4018 + 0.2612\ii  &  0.2689 + 0.0968\ii & -0.2733 + 1.2440\ii &  -0.8719 - 0.3476\ii\\
  0.4238 - 0.2379\ii &   0.3380 + 0.6318\ii  & -0.3410 + 0.3895\ii & -0.8719 - 0.3476\ii &  -0.4230 + 1.2122\ii
  \end{bsmallmatrix*}.
\end{align*}
The singular values~\cite{gv} of $F_5$ 
are  
$\{1.9376, \  1.9376, \ 1.9376,\  1.9376, \ 1 \}$. 
Hence, the next doubling step  breaks down and the DCT  
is required to carry the DA forward. 
\end{Example}

\subsection{Three-recursion remedy}

This subsection is devoted to resolve the issue that the DCT fails. 
Especially, one may apply the three-recursion remedy from this section when two step reversions occur 
with some complex eigenvalues for $H$. 

Let $Z = Z^{\T} \in \mathbb{C}^{n \times n}$ (which may be chosen randomly) 
and $I_n + F_k^{\T} Z$ be nonsingular.   
Write $P_k=(I_n + F_k Z)^{-1}E_k$, $G_k=(I_n + F_k Z)^{-1}F_k^{\T}$ and  
$H_k = (F_k +Z) - E_k^{\T} Z (I_n + F_k Z)^{-1}E_k$. 
The following lemma shows how we transform the two recursions for $E_k$ and $F_k$ to three. 

\begin{lemma}\label{lm:three_recursion_initial}
For the decomposition \eqref{eq:eigen_decomp} it holds that 
\begin{align}\label{eq:doubling_pure_three}
	&
	\begin{bmatrix}
		P_k & \mathbf 0 \\ \\ H_k & I_n 
	\end{bmatrix}
	\begin{bmatrix}
		I & \mathbf 0 \\ \\  -Z& I_n
	\end{bmatrix}
	\begin{bmatrix}
		X_1&W_{1,\omega}& \overline{X}_2 & Z_{1,\omega}\\ \\
		X_2&W_{2,\omega}& \overline{X}_1& Z_{2,\omega}
	\end{bmatrix} \nonumber
	\\
	=&
	\begin{bmatrix}
		I_n& G_k \\ \\ \mathbf 0 & P_k^{\T}  
	\end{bmatrix}
	\begin{bmatrix}
		I & \mathbf 0\\ \\ -Z& I_n
	\end{bmatrix}
	\begin{bmatrix}
		X_1&W_{1,\omega}& \overline{X}_2 & Z_{1,\omega}\\ \\
		X_2&W_{2,\omega}& \overline{X}_1&Z_{2,\omega}
	\end{bmatrix}
	\widetilde S_{\alpha}^{2^k},
\end{align}
where $X_1, X_2, W_{1,\omega}, W_{2,\omega}, Z_{1,\omega}, Z_{2,\omega}$ and $\widetilde {S}_{\alpha}^{2^k}$ 
are defined as in \eqref{eq:doubling_pure}. 
\end{lemma}

\begin{proof}
Define $\Phi = \begin{bmatrix}
	(I_n + F_k Z)^{-1} & \mathbf 0\\ \\ -E_k^{\T} Z(I_n + F_k Z)^{-1} & I_n
\end{bmatrix}$, then we deduce that 
\begin{align*}
	\Phi \begin{bmatrix}
	E_k & \mathbf 0\\ \\ F_k & I_n
\end{bmatrix}
\begin{bmatrix}
	I_n & \mathbf 0\\ \\ Z& I_n
\end{bmatrix}
= 
\begin{bmatrix}
	P_k& \mathbf 0\\ \\ H_k & I_n
\end{bmatrix}, 
& \qquad 
\Phi \begin{bmatrix}
	I_n & \overline F_k  \\ \\ \mathbf 0 &  \overline E_k
\end{bmatrix}
\begin{bmatrix}
	I_n & \mathbf 0\\  \\ Z& I_n
\end{bmatrix}
= 
\begin{bmatrix}
	I_n & G_k \\ \\  \mathbf 0 & P_k^{\T}
\end{bmatrix}. 
\end{align*}
With $\begin{bmatrix}
	I_n & \mathbf 0\\ \\ Z & I_n
\end{bmatrix}^{-1} = \begin{bmatrix}
	I_n & \mathbf 0 \\ \\ -Z & I_n
\end{bmatrix}$, the result follows from \eqref{eq:doubling_pure}.
\end{proof}

Since $F_k^{\T} = F_k$ and $Z^{\T} = Z$, we have $G_k^{\T} = G_k$ and $H_k^{\T} = H_k$. 
Applying the doubling algorithms \cite{lx06} for CARE and DARE, provided that 
$(I_n - G_{k+j-1} H_{k+j-1})^{-1}$  are well-defined for $j\geq 1$,  we formulate  
the three recursions for $P_{k+j}, G_{k+j}$ and $H_{k+j}$  as below:
\begin{equation}\label{eq:three_recursions}
	\begin{aligned}
		P_{k+j} &= P_{k+j-1} (I_n - G_{k+j-1} H_{k+j-1})^{-1} P_{k+j-1},\\ 
		G_{k+j} &= G_{k+j-1} + P_{k+j-1} (I_n - G_{k+j-1} H_{k+j-1})^{-1} G_{k+j-1} P_{k+j-1}^{\T},\\
		H_{k+j} &= H_{k+j-1} + P_{k+j-1}^{\T} H_{k+j-1}  (I_n - G_{k+j-1} H_{k+j-1})^{-1}  P_{k+j-1}, 
	\end{aligned}
\end{equation}
where $G_{k+j}^{\T} = G_{k+j}$ and $H_{k+j}^{\T} = H_{k+j}$. 
It is worthwhile to point that when $I_n - G_{k+j} H_{k+j}$ is singular or ill-conditioned, 
we can always randomly choose some other $Z^{\T} = Z \in \mathbb{C}^{n\times n}$ and construct $\Psi\in \mathbb{C}^{2n\times 2n}$ 
such that 
\begin{align*}
	\Psi \begin{bmatrix}
		P_{k+j} &\mathbf 0 \\ \\ H_{k+j} & I_n
	\end{bmatrix}\begin{bmatrix}
		I_n & \mathbf 0 \\ \\ Z& I_n
	\end{bmatrix}
	= 
	\begin{bmatrix}
		\widetilde P_{k+j} & \mathbf 0\\ \\ \widetilde H_{k+j} & I_n
	\end{bmatrix},
	& \qquad
	\Psi \begin{bmatrix}
		I_n &  G_{k+j} \\  \\ \mathbf 0  & P_{k+j}^{\T}
	\end{bmatrix}\begin{bmatrix}
		I_n & \mathbf 0 \\ \\ Z& I_n
	\end{bmatrix}
	= 
	\begin{bmatrix}
		I_n &	\widetilde G_{k+j} \\  \\ \mathbf 0 & \widetilde P_{k+j}^{\T}
	\end{bmatrix}.
\end{align*}

Provided that $I_n - G_{k+j}H_{k+j}$ are well-conditioned for all $j \geq0$, the following two theorems 
demonstrate the convergence of the three recursions specified in \eqref{eq:three_recursions}.

\begin{theorem}\label{thm:convergence_three_recursions}
	Upon the assumption in Theorem~\ref{convergence},	
	it holds that $\lim_{k\to \infty} P_{k} =0$ and $\lim_{k\to \infty} H_{k} = Z-X_2X_1^{-1}$, 
	both converging quadratically.  
\end{theorem}  
\begin{proof}
	The results follow from the fact 
	\begin{align*}
		&	\begin{bmatrix}
		P_k& \mathbf 0\\ \\ H_k & I_n
	\end{bmatrix} \begin{bmatrix}
		X_1 & \overline{X}_2 \\ \\ X_2 - Z X_1 & \overline{X}_1 - Z \overline X_2 
	\end{bmatrix} 
	\\
	=&
	\begin{bmatrix}
		I_n & G_k \\ \\ \mathbf 0 & P_k^{\T}
	\end{bmatrix}\begin{bmatrix}
		X_1 & \overline{X}_2 \\ \\ X_2 - Z X_1 & \overline{X}_1 - Z \overline X_2
	\end{bmatrix}
	\begin{bmatrix}
		S_{\alpha}^{2^k} &\\ \\ & \overline{S}_{\alpha}^{-2^k}
	\end{bmatrix}
\end{align*}
and $\lim_{k\to \infty} S_{\alpha}^{2^k}=0$. We omit the details, as in \cite[Corollary~3.2]{lx06}.  
\end{proof}

\begin{theorem}\label{thm:convergence_three_recursions_pure}
	Under the assumption in Theorem~\ref{thm:conv_pure},	
	it holds that $\lim_{k\to \infty} P_{k} =0$ and $\lim_{k\to \infty} H_{k} = Z-X_2X_1^{-1}$, 
	both converging linearly.  
\end{theorem}

\begin{proof}
	By \eqref{eq:doubling_pure_three} and similar to the proof of Theorem~\ref{thm:conv_pure}, 
we obtain the result. 
\end{proof}

\section{Numerical Results}
We illustrate the performance of the DA  
with some test  examples, three of which from 
discretized  Bethe-Salpeter equations and one generated by the \verb|randn| command in MATLAB. 
We also apply  \verb|eig|   in MATLAB  
(as in   \verb|eig|$(H)$ and \verb|eig|$(\Gamma H, \Gamma)$) and Algorithm~1 in~\cite{yang:16} to the test examples for comparison.  
Computing  \verb|eig|$(\Gamma H, \Gamma)$ is based on the equivalence of  
$Hx = \lambda x$  and $\begin{bmatrix}A& B\\ \overline{B}&\overline{A}\end{bmatrix}x = 
\lambda \begin{bmatrix}I_n&\mathbf 0\\ \mathbf 0& -I_n\end{bmatrix}x$.
No DCT or three-recursion remedy was required. 
All algorithms are implemented in  MATLAB 2012b on a 64-bit PC  with an Intel Core~i7 processor at 3.4 GHz and 8G RAM.

\begin{Example}\rm \label{ex1}
We consider three examples from the discretized  Bethe-Salpeter equations for  naphthalene (\ce{C10H8}), 
gallium arsenide (\ce{GaAs}) and  boron nitride (\ce{BN}).  
The dimensions of the corresponding $H$ associated with \ce{C10H8}, 
\ce{GaAs} and \ce{BN} are respectively $64$, $256$ and $4608$.  
All  eigenpairs of $H$ are computed.  

Using \verb|eig|$(H)$ as the baseline for comparison, we present 
 the relative accuracy of the computed eigenvalues and the execution time (eTime)  of  the other three algorithms, all averaged over 50 trials. 
For the relative accuracy, we compute    
$\mathrm{prec} =\log_{10} [\max_{j}| (\lambda_j - \widehat{\lambda}_j)/\lambda_j |]$ 
where $\lambda_j$ and $\widehat{\lambda}_j$ are the computed eigenvalues by the 
\verb|eig|$(H)$ command and  one of the methods, respectively. 
The residuals       
\[ 
	\frac{\|H-[X,\, \Pi\overline{X}]\diag(S, \overline{S}) [X,\,\Pi\overline{X}]^{-1}\|_F}{\|H\|_F}, \ \ \ 
	\frac{\|Y^{\HH}HX - \Lambda\|_F}{\|H\|_F}, 
\] 
respectively for the DA,  
\texttt{eig}$(\Gamma H, \Gamma)$ and \cite[Algorithm~1]{yang:16} are displayed,    
with $Y$ and $X$ being respectively the left and right eigenvector matrices  
and $\Lambda$ the  diagonal matrix containing the eigenvalues of $H$
(please refer to~\cite{yang:16} for details). Also, the numbers of iterations 
required for doubling averaged over 50 trails  are presented. 
It is  worthwhile to point out  that for the DA all $\alpha$'s 
in the 50 trails are generated by the function \texttt{randn}.
The results are tabulated in Table~\ref{table}.

\begin{table}[H]
\footnotesize
\centering
\begin{tabular}{c|c|c|c}
\hline 
\multicolumn{4}{c}{\ce{C10H8}}\\
\hline 
&DA & algorithm 1 in~\cite{yang:16}   &  \verb|eig|$(\Gamma H, \Gamma)$\\
\hline 
prec & $-13.97$ & $-13.92$ & $-13.95$  \\
residual &$8.14\times10^{-16}$& $2.60\times10^{-15}$ & $1.71\times10^{-15}$  \\
eTime & $7.958\times10^{-1}$ & $5.764\times10^{-1}$ & $3.792\times10^{-1}$  \\
iteration &  $6.84$ & $-$   &$-$ \\
\hline 
\multicolumn{4}{c}{\ce{GaAs}}\\
\hline 
&DA & algorithm 1 in~\cite{yang:16}   &  \verb|eig|$(\Gamma H, \Gamma)$\\
\hline 
prec & $-13.74$ & $-13.54$ & $-13.75$\\
residual & $6.86\times10^{-16}$ & $6.33\times 10^{-15}$ & $5.07\times10^{-15}$ \\
eTime & $5.881\times10^{-1}$ & $3.587\times10^{-1}$  & $3.533\times10^{-1}$\\
iteration  & $8.46$ & $-$ & $-$\\
\hline 
\multicolumn{4}{c}{\ce{BN}}\\
\hline 
&DA & algorithm 1 in~\cite{yang:16}   &  \verb|eig|$(\Gamma H, \Gamma)$\\
\hline 
prec & $-13.11$ & $-13.12$ & $-13.04$\\
residual & $7.50\times 10^{-16}$ & $2.54\times 10^{-14}$ & $1.63\times 10^{-14}$\\
eTime &  $6.610\times10^{-1}$ & $4.754\times10^{-1}$ & $4.843\times10^{-1}$\\
iteration &  $7.44$  & $-$ & $-$\\
\hline 
\end{tabular}
\caption{Numerical results for Example~\ref{ex1}}
\label{table}
\end{table}

Table~\ref{table} demonstrates that all three methods produce comparable results in terms of the relative accuracy. The DA spends 
slightly more time than the other methods but produces more accurate 
solutions with smaller residuals. 

\end{Example}

\begin{Example}\rm \label{ex2}
The test  example,   randomly generated by the command \verb|randn| in MATLAB, is designed to 
illustrate the structure-preserving property    of the DA, a distinct feature of our method. 
The defining matrices are 
\begin{align*}
H=\begin{bmatrix}
\ \ A & \ \ B \\ \\ -\overline{B} & -\overline{A}
\end{bmatrix}, \qquad 
A = \begin{bmatrix}
A_1 &   &   \\   & A_2 &   \\   &  & A_3
\end{bmatrix}, & \qquad 
B = \begin{bmatrix}
B_1 &   &   \\   & B_2 &   \\   &  & B_3
\end{bmatrix}
\end{align*}
with 
\begin{align*}
	A_1 & = \begin{bmatrix*}[l]
 2.6361               &  \hm 1.0378\times10^{1}    & \hm  5.0751\times10^{-2}  \\ 
 1.0378\times10^{1}   &  \hm 5.2431\times10^{-2}   & -4.6067\times10^{-1} \\  
 5.0751\times10^{-2}  & -4.6067\times10^{-1}   & -1.6892\times10^{-2} 
\end{bmatrix*}, \\
A_2 & =\begin{bmatrix*}[l]
-4.0549\times10^{-1}    & -3.7710+2.7569 \ii       \\
-3.7710-2.7569 \ii        & -4.0549\times10^{-1}  
\end{bmatrix*}, \\
A_3 &=\begin{bmatrix*}[l]
3.6378\times10^{-1}            & 2.7293\times10^{-1} + 3.5908 \ii \\
2.7293\times10^{-1}-3.5908 \ii   & 3.6378\times10^{-1}
\end{bmatrix*}, \\
B_1 &=\begin{bmatrix*}[l]
-2.6361              &-1.0375\times10^{1}    &-5.1181\times10^{-2} 	 \\
-1.0375\times10^{1}  &-5.3457\times10^{-2}   &\hm 5.0988\times10^{-1}   \\
-5.1181\times10^{-2} &\hm 5.0988\times10^{-1}   &\hm 4.2022\times10^{-3} 
\end{bmatrix*}, \\
B_2 &=\begin{bmatrix}
1.2343\times10^{-1}-3.8788\ii\times10^{-1} & 3.7566-2.7464\ii   \\
3.7566-2.7464\ii                           &	4.0704\times10^{-1}+6.0156\ii\times10^{-5}
\end{bmatrix}, \\
B_3 &=\begin{bmatrix*}[l]
 \hm 3.6148\times10^{-1}-5.5211 \ii \times10^{-2} &  -2.7152\times10^{-1}-3.5722\ii \\
-2.7152\times10^{-1} -3.5722\ii               &  -3.6567\times10^{-1}+5.9265\ii\times10^{-5}
\end{bmatrix*}. 
\end{align*}
The spectrum of $H$ is 
\begin{align*}
\lambda(H) &= 
	\begin{array}[t]{c@{\hspace{0pt}}ll}
		 \{& \pm 4.1204\times10^{-3},  \quad \pm 4.1204\times10^{-3}, &\pm 4.1204\times10^{-3}, \\
		&\pm 4.0549\times10^{-1}\pm 5.9927\ii\times10^{-5},    &  \pm 3.6378 \times10^{-1} \pm 5.8959\ii \times10^{-5} \}.
	\end{array}
\end{align*}
Note that the algebraic and the 
geometric multiplicities of \mbox{$\pm 4.1204\times10^{-3}$} are $3$ and $1$, respectively.  The DA,  \verb|eig|$(H)$ and 
 \verb|eig|$(\Gamma H, \Gamma)$ produce the eigenvalues $\lambda_{D}$,  $\lambda_{E}$  and $\lambda_{Ge}$ respectively: 
\begin{align*}
	\lambda_{D} &=
	\begin{array}[t]{c@{\hspace{0pt}}ll}
		 \{& \pm 4.1092\times10^{-3},  \quad \pm 4.1092\times10^{-3}, &\pm 4.1092\times10^{-3}, \\
		&\pm 4.0549\times10^{-1}\pm 5.9927\ii\times10^{-5},    &  \pm 3.6378 \times10^{-1} \pm 5.8959\ii \times10^{-5} \},
	\end{array}
	\\
	\lambda_{E} &=
	\begin{array}[t]{c@{\hspace{0pt}}rr@{\hspace{0pt}}l}
		 \{& 4.1137\times10^{-3} - 1.1615\ii\times10^{-5},  & 4.1136\times10^{-3}+ 1.1614\ii\times10^{-5}&,    \\
			& 4.1338\times10^{-3} + 1.2681\ii\times10^{-9},  &                                                    \\
			& -4.1136\times10^{-3} - 1.1649\ii\times10^{-5}, &  -4.1136\times10^{-3} + 1.1650\ii\times10^{-5}&, \\
			& -4.1338\times10^{-3} - 1.3011\ii\times10^{-9},&&                                                     \\
		& \pm 4.0549 \times10^{-1} \pm  5.9927\ii\times10^{-5}, & \pm 3.6378\times10^{-1} \pm 5.8959\ii\times10^{-5}&\},
	\end{array}
	\\
	\lambda_{Ge} &= 
	\begin{array}[t]{c@{\hspace{0pt}}rr@{\hspace{0pt}}l}
		\{& 4.1272\times10^{-3}-1.1919\ii \times10^{-5}, &  4.1272\times10^{-3}-1.1919\ii \times10^{-5}&,\\
			& 4.1272\times10^{-3}-1.1919\ii \times10^{-5}, &&\\
			& -4.1272\times10^{-3} + 1.1851\ii \times10^{-5},  &  -4.1272 \times10^{-3}+ 1.1851\ii \times10^{-5}&, \\
			& -4.1272 \times10^{-3} + 1.1851\ii \times10^{-5},&& \\
		& \pm 4.0549 \times10^{-1}  \pm 5.9927\ii \times10^{-5}, & \pm3.6378  \times10^{-1}  \pm 5.8959\ii \times10^{-5}&\}.
	\end{array}
\end{align*}
Although all three  methods produce computed eigenvalues of low relative accuracy, with $prec_D= -2.5680$, $prec_E=-2.4862$ and $prec_{Ge}=-2.4764$, 
the DA  preserves the  distinct eigen-structure of $H$. All eigenvalues from DA appear in  quadruples  
$\{\lambda, \overline{\lambda}, -\lambda, -\overline{\lambda}\}\subseteq\lambda(H)$, unless when $\Im(\lambda)=0$ then in pairs 
$\{\lambda, -\lambda\}\subseteq\lambda(H)$. The low accuracy (in the order of $\pm 4.1204\times10^{-3}$) of the computed eigenvalues from the methods can be attributed to the defective eigenvalues. 
Note that Algorithm~1 in~\cite{yang:16} failed because the required assumption $\Gamma H>0$  is not satisfied.  
\end{Example}

\section{Conclusions} 
In this paper, we propose a doubling algorithm for the discretized Bethe-Salpeter eigenvalue problem, 
where the Hamiltonian-like matrix $H$ is 
firstly transformed to a symplectic pair with  special structure then $E_k = E_k^{\HH}$  and $F_k=F_k^{\T}$  
are computed iteratively. Theorems are proved on the quadratic convergence of the algorithm if no 
purely imaginary eigenvalues exist 
(and linear convergence otherwise). 
The simple double-Cayley transform is designed to deal with any potential breakdown when  
$1$ is in or close to $\sigma(F_k)$ for some $k$. 
We also prove   that at most two  steps of retrogression occur (for complex eigenvalues of $H$, but  
none for real ones). In addition, a three-recursion remedy is put forward when 
the double-Cayley transform fails. 
Numerical examples have been presented to illustrate the efficiency 
and the distinct structure-preserving nature of the doubling method. The optimal choice of $\alpha$ 
and the removal of the invertibility assumption of $X_1$ 
(or $[
	X_1, \Psi_{11} 
]$ if purely imaginary eigenvalues exist)  will be left for future research. 

\section*{Acknowledgements} 
 We thank Prof. Ren-Cang Li for his kindness in providing three test problems in Example~4.1.

\appendix
\section{Useful Lemmas}
The following lemmas are required in Section~\ref{doublesec}.

\begin{lemma}\label{lemma-tanh}
Given $\omega, \zeta \in \mathbb{R}$, it holds that 
\begin{enumerate}
  \item [{\em (a)}]  $|\tanh(-\omega+\ii\zeta)|^2=|\tanh(\omega+\ii\zeta)|^2
=[\ee^{2\omega}+\ee^{-2\omega}-2\cos(2\zeta)][\ee^{2\omega}+\ee^{-2\omega}+2\cos(2\zeta)]^{-1}$; 
  \item [{\em (b)}] $|\tanh(\omega+\ii\zeta)|^2<1$ when $\cos(2\zeta)>0$; and 
  \item [{\em (c)}] for  $\cos(2\zeta)>0$,  $|\tanh(\omega+\ii\zeta)|^2$ 
  is monotonically  nondecreasing with respect to $\omega$ when $\omega\geq 0$, and 
  monotonically  nonincreasing otherwise.   
\end{enumerate}
\end{lemma}
\begin{proof}
Simple computations lead to the two  results (a) and (b),  
and we omit the details here. 
For (c), we have  $\partial |\tanh(\omega+\ii\zeta)|^2/\partial\omega= [8(\ee^{2\omega}-\ee^{-2\omega})\cos(2\zeta)]
[(\ee^{2\omega}+\ee^{-2\omega}+2\cos(2\zeta))^2]^{-1}$. Since $\cos(2\zeta)>0$,  
the result follows. 
\end{proof}

\begin{lemma}\label{lemma-distance}
Define $f(\xi)=(\xi-\tau)^2 + \xi^2$, then for $0\leq\xi\leq\frac{\tau}{2}$, 
we have
\begin{enumerate}
\item [{\em (a)}] $f(\xi)=f(\tau-\xi)$;
\item [{\em (b)}] $f(\xi)\geq f(\eta) \geq \frac{\tau}{\sqrt{2}}$ for all $\eta$ 
                  with $\frac{\tau}{2}\geq \eta \geq \xi$; and 
\item [{\em (c)}] $f(\xi)\geq f(\eta) \geq \frac{1}{\sqrt{2}}$ for all $\eta$  with 
                  $\tau-\xi \geq \eta \geq \frac{\tau}{2}$.  
\end{enumerate}
\end{lemma}

\begin{proof}
From the fact that  $(\xi, \xi)$ and $(1-\xi, 1-\xi)$ 
are two  symmetrical points with respect to the line $g(\omega)=-\omega +\tau$, the result follows with details omitted.
\end{proof}

\end{document}